\documentclass[leqno,letterpaper]{amsart}
\usepackage{amsmath}
\usepackage{rotating}
\usepackage{mathrsfs}
\usepackage{amsfonts}
\usepackage{amssymb}
\usepackage{multicol}
\usepackage{amscd}
\usepackage{amsthm}
\usepackage{latexsym}
\usepackage{amsbsy}
\usepackage{lscape}
\usepackage{verbatim}
\newtheorem{theo}{Theorem}[section]
\newtheorem{lemm}[theo]{Lemma}
\newtheorem{prop}[theo]{Proposition}
\newtheorem{coro}[theo]{Corollary}
\theoremstyle{definition}
\newtheorem{defi}[theo]{Definition}
\newtheorem{example}[theo]{Example}
\newtheorem{nota}[theo]{Notation}

\theoremstyle{remark}
\newtheorem{rema}[theo]{Remark}
\newcommand{\bahram}{\textbf{(CR)}}
\begin{document}
\title[$K$-theory of Furstenberg transformation group $C^*$-algebras]
{$K$-theory of Furstenberg transformation group $C^*$-algebras
%(\textnormal{Only a Draft})
}
\author[Kamran Reihani]{Kamran Reihani}
\address{Department of Mathematics\\University of Kansas\\Lawrence, KS 66045-7594}
\email{reihani@math.ku.edu}
%\date{March 27, 2010}
\maketitle

\begin{abstract}
The paper studies the $K$-theoretic invariants of the crossed product $C^{*}$-algebras associated with an important family of homeomorphisms of the tori $\Bbb{T}^{n}$ called {\em Furstenberg transformations}.  Using the Pimsner-Voiculescu theorem, we prove that given $n$, the $K$-groups of those crossed products, whose corresponding $n\times n$ integer matrices are unipotent of maximal degree, always have the same rank $a_{n}$.  We show using the theory developed here, together with two computing programs - included in an appendix - that a claim made in the literature about the torsion subgroups of these $K$-groups is false. Using the representation theory of the simple Lie algebra $\frak{sl}(2,\Bbb{C})$, we show that, remarkably, $a_{n}$ has a combinatorial significance.  For example, every $a_{2n+1}$ is just the number of ways that $0$ can be represented as a sum of integers between $-n$ and $n$ (with no repetitions).  By adapting an argument of van Lint (in which he answered a question of Erd\"{o}s), a simple, explicit formula for the asymptotic behavior of the sequence $\{a_{n}\}$ is given.  Finally, we describe the order structure of the $K_{0}$-groups of an important class of Furstenberg crossed products, obtaining their complete Elliott invariant using classification results of H. Lin and N. C. Phillips. 
\end{abstract}

%%%%%%%%%%%%%%%%%%

\tableofcontents

%%%%%%%%%%%%%%%%%%

\vspace{.5cm}

\section{Introduction}

\noindent Furstenberg transformations were introduced in \cite{hF61}
as the first examples of homeomorphisms of the tori, which under some necessary and sufficient 
conditions are minimal and uniquely ergodic. In some sense, they generalize
the irrational rotations on the circle. They also appear in certain applications of ergodic theory to number theory 
(e.g. in Diophantine approximation \cite{hF81}), and sometimes are called \emph{skew product transformations} or \textit{compound skew translations} of the tori. The terminology ``Furstenberg transformation group $C^*$-algebra" is what 
we would like to use in this paper to call the crossed products associated with Furstenberg transformations, and we will denote them by ${\mathscr F}_{\theta,{\boldsymbol f}}$. There have been several contributions to the computations of $K$-theoretic invariants for some examples of these $C^*$-algebras in the literature (see \cite{rJ86,kK00,pM93,ncp07,sW02} to name a few). However, a more general study of such invariants for these $C^*$-algebras has not been available to the best of our knowledge.\\

\begin{rema}
In independent (unpublished) work \cite{rJ86}, R. Ji studied the $K$-groups of the $C^{*}$-algebras ${\mathscr F}_{\theta,{\boldsymbol f}}$ (denoted by $A_{F_{f,\theta}}$ in there) associated with the descending affine Furstenberg transformations (denoted by ${F_{f,\theta}}$ in there) on the tori.
He comments that ``explicitly computing the $K$-groups of $\mathcal{C}(\Bbb T^n)\rtimes_K\Bbb Z$ [$A_{F_{f,\theta}}$ for $\theta=0$] is still not an easy matter''. Moreover, he gives no information about the ranks of the $K$-groups or order structure of $K_0$ in general, which are studied in the present paper.  As we shall see in Remark \ref{rema:2} below, the claim that he makes about the form of the torsion subgroup of $K_*({\mathscr F}_{\theta,{\boldsymbol f}})$ is unfortunately not correct.
\end{rema}
~\\
\noindent From the $C^*$-algebraic point view, when a Furstenberg 
transformation is minimal and uniquely ergodic, the associated transformation group $C^*$-algebra is simple
and has a unique tracial state with a dense tracial range of the $K_0$-group in the real line. Because of this, these $C^{*}$-algebras fit well into the classification program of G. Elliott by finding their $K$-theoretic invariants. In fact, in the class of transformation group $C^*$-algebras of uniquely ergodic minimal homeomorphisms on 
infinite compact metric spaces, $K$-theory is a complete invariant. More precisely, suppose that 
$X$ is an infinite compact metric space with finite covering dimension and \text{$h:X 
\rightarrow X$} is a uniquely ergodic minimal homeomorphism, and put $A:=\mathcal{C}(X)\rtimes_h 
\Bbb{Z}$. Let $\tau$ be the trace induced by the unique invariant probability measure. Then $\tau$
is the unique tracial state on $A$. Let $\tau_*:K_0(A)\rightarrow\Bbb R$ be the induced homomorphism 
on $K_0(A)$ and assume that $\tau_*K_0(A)$ is dense in $\Bbb{R}$. Then the $4$-tuple
$$(K_0(A),K_0(A)_{+},[1_A],K_1(A))$$
is a complete algebraic invariant (called the \emph{Elliott invariant} of 
$A$) \cite[Corollary 4.8]{LP10}. In this case, $A$ has stable rank one, real rank zero and tracial 
topological rank zero in the sense of H. Lin \cite{hL00}. The order on $K_0(A)$ is also 
determined by the unique trace $\tau$, in the sense that an element $x\in K_0(A)$ is positive if
and only if either $x=0$ or $\tau_*(x)>0$ \cite{qLP97,ncP07}. This implies, in particular, that the torsion subgroup of 
$K_0(A)$ contributes nothing interesting to the order information. In other words,
the order on $K_0(A)$ is determined by the order on the free part.
We will study the order structure of $K_0({\mathscr F}_{\theta,{\boldsymbol f}})$ in 
Section \ref{order}.\\

\noindent In order to compute the $K$-groups of a crossed product of the form $\mathcal{C}(\Bbb T^n)\rtimes_\alpha\Bbb Z$ in general, we make use of the algebraic properties of  $K_*(\mathcal{C}(\Bbb{T}^n))$ in Section \ref{sec:general}. More precisely, $K_*(\mathcal{C}(\Bbb{T}^n))$
is an exterior algebra over $\Bbb{Z}^n$ with a certain natural basis, 
and the induced automorphism $\alpha_*$ on $K_*(\mathcal{C}(\Bbb{T}^n))$
is in fact a ring automorphism, which makes computations much easier. In 
fact, it is shown in Theorem \ref{theo:1} that the problem of finding 
the $K$-groups of the transformation group $C^*$-algebra of a homeomorphism of the $n$-torus 
is completely computable in the sense that one only needs to calculate 
the kernels and cokernels of a finite number of integer matrices. These $K$-groups are finitely generated with the same rank
(see Corollary \ref{coro:3}). In the special case of Anzai transformation group $C^*$-algebras $\mathscr{A}_{n,\theta}$ associated with Anzai transformations on the $n$-torus, we denote this common rank by $a_n$, which we will study in detail in this paper. It is proved in Theorem \ref{rank} that $a_n$ is the common rank of the $K$-groups of a larger class of transformation group $C^*$-algebras, including the $C^*$-algebras associated with Furstenberg transformations on $\Bbb{T}^n$. We describe $a_n$ as the constant term in a certain Laurent polynomial (Theorem \ref{theo:6}). Then we study the combinatorial properties of the sequence $\{a_n\}$, which leads to a simple asymptotic
formula.\\

%An explicit formula for the torsion subgroups of the $K$-groups of these 
%algebras seems much more challenging to find.

\noindent To present the results and proofs of this paper we 
need some definitions about transformations on the tori and the corresponding 
\text{$C^*$-crossed products.} Throughout this paper, $\mathbb{T}^n$ denotes the $n$-dimensional torus with coordinates 
$(\zeta_1,\zeta_2,\ldots,\zeta_n)$.

\begin{defi}
An \textbf{affine transformation} on $\Bbb{T}^n$ is given by
\begin{equation*}
\beta(\zeta_1,\zeta_2,\ldots,\zeta_n)=(e^{2 \pi i t_1} \zeta_1^{b_{11}}\ldots 
\zeta_n^{b_{1n}},
e^{2 \pi i t_2} \zeta_1^{b_{21}}\ldots \zeta_n^{b_{2n}},\ldots,
e^{2 \pi i t_n} \zeta_1^{b_{n1}}\ldots \zeta_n^{b_{nn}}),
\end{equation*}
where $\boldsymbol{t}:=(t_1,t_2,\ldots,t_n)\in(\Bbb{R}/\Bbb Z)^n $ and 
$\mathsf{B}:=[b_{ij}]_{n \times n}\in
\mathrm{GL}(n,\Bbb{Z})$. We identify the pair $(\boldsymbol{t},\mathsf{B})$ 
with $\beta$.
\end{defi}

\noindent Note that any automorphism of $\Bbb{T}^n$ followed by a rotation can be 
expressed in such a fashion. The set of affine transformations on $\Bbb{T}^n$ form a group 
$\mathrm{Aff}(\Bbb{T}^n)$, which can be identified with the semidirect product $(\Bbb{R}/\Bbb Z)^n 
\rtimes\mathrm{GL}(n,\Bbb{Z})$. More precisely, for two affine transformations 
$\beta=(\boldsymbol{t},\mathsf{B})$
and $\beta'=(\boldsymbol{t}',\mathsf{B'})$ on $\Bbb{T}^n$, we have
$$\beta\circ\beta'=(\boldsymbol{t}+\mathsf{B}\boldsymbol{t}',
\mathsf{B}\mathsf{B}')
~~\text{and}~~\beta^{-1}=(-\mathsf{B}^{-1}\boldsymbol{t},\mathsf{B}^{-1}).$$
(In the expression $\mathsf{B}\boldsymbol{t}$, $\boldsymbol{t}$ is a  column 
vector, but for convenience we write it as a row vector.)\\

\noindent We remind the reader of an important fact before giving the next definition. Consider homotopy classes 
of continuous functions from $\mathbb{T}^n$ to $\Bbb{T}$. It is well known that in each class there is a unique ``linear" function $(\zeta_1,\ldots,\zeta_n)\mapsto\zeta_1^{b_1}\ldots \zeta_n^{b_n}$ for some $b_1,\ldots,b_n \in\mathbb{Z}$. More precisely, every continuous function $f:\Bbb {T}^n\to\Bbb T$ can be written as
$$f(\zeta_1,\ldots,\zeta_n)=e^{2\pi i g(\zeta_1,\ldots,\zeta_n)}\zeta_1^{b_1}\ldots \zeta_n^{b_n},$$
for some continuous function $g:\Bbb {T}^n\to\Bbb R$ and unique integer exponents $b_1,\ldots,b_n$. In particular, 
the cohomotopy group $\pi^1(\Bbb T^n)$ is isomorphic to $\Bbb Z^n$.
Following \cite[p. 35]{hF81}, we denote the exponent $b_i$, which is uniquely determined by 
the homotopy class of $f$, as $b_i=A_i[f]$.\\

\begin{defi}\label{furst-anzai}
 We define the following transformations in accordance with \cite{rJ86}.
\begin{itemize}
\item [\textnormal{(a)}]
A \textnormal{\textbf{Furstenberg transformation}} $\varphi_{\theta,{\boldsymbol f}}$ on 
$\Bbb{T}^n$ is given by
\begin{equation*}
\varphi^{-1}_{\theta,{\boldsymbol f}}(\zeta_1,\zeta_2,\ldots,\zeta_n)=\left(e^{2 \pi i \theta} 
\zeta_1,f_1(\zeta_1)\zeta_2,f_2(\zeta_1,\zeta_2)\zeta_3,
\ldots, f_{n-1}(\zeta_1,\ldots,\zeta_{n-1})\zeta_n \right),
\end{equation*}
where $\theta$ is a real number, each
$f_i:\Bbb{T}^i \rightarrow\Bbb{T}$ is a continuous function with
$A_i[f_i]\neq 0$ for $i=1,\ldots,n-1$, and ${\boldsymbol f}=(f_1,\ldots,f_{n-1})$.
\item [\textnormal{(b)}]
An \textnormal{\textbf{affine Furstenberg transformation}} $\alpha$ on $\Bbb{T}^n$ is 
given by
$$\alpha^{-1}(\zeta_1,\zeta_2,\ldots,\zeta_n)=(e^{2 \pi i \theta} \zeta_1,\zeta_1^{b_{12}}\zeta_2,
\zeta_1^{b_{13}}\zeta_2^{b_{23}} \zeta_3,\ldots
,\zeta_1^{b_{1n}}\zeta_2^{b_{2n}}\ldots \zeta_{n-1}^{b_{n-1,n}}\zeta_n),$$
where $\theta$ is a real number and the exponents $b_{ij}$
are integers and $b_{i,i+1}\neq 0$ for $i=1,\ldots,n-1$.
\item [\textnormal{(c)}]
An \textnormal{\textbf{ascending Furstenberg transformation}} $\alpha$ on $\Bbb{T}^n$ 
is given by
$$\alpha^{-1}(\zeta_1,\zeta_2,\ldots,\zeta_n)=(e^{2 \pi i \theta} \zeta_1,\zeta_1^{k_1} \zeta_2,
\zeta_2^{k_2} \zeta_3,\ldots
, \zeta_{n-1}^{k_{n-1}} \zeta_n),$$
where $\theta$ is a real number and the exponents $k_i$ are nonzero
integers and $k_i \mid k_{i+1}$ for $i=1,\ldots,n-2$.
\item [\textnormal{(d)}]
In \textnormal{(c)}, if $k_i=1$ for $i=1,\ldots,n-1$, the transformation is 
called an \textnormal{\textbf{Anzai transformation}} $\sigma_{n,\theta}$ on $\Bbb{T}^n$. Thus it is 
given by
$$\sigma_{n,\theta}^{-1}(\zeta_1,\zeta_2,\ldots,\zeta_n)=(e^{2 \pi i \theta} \zeta_1,\zeta_1 \zeta_2,\ldots,\zeta_{n-1} 
\zeta_n),$$
where $\theta$ is a real number. We usually drop the indices $n$ and $\theta$ and write only $\sigma$ for more convenience.
\end{itemize}
\end{defi}

\noindent \noindent Note that one can easily verify that $\varphi_{\theta,{\boldsymbol f}}$ is a homeomorphism. Also in the above definition, we have converted
``descending", which is used in \cite[Definition 2.16]{rJ86}, to ``ascending" since the order of coordinates there is opposite to ours. \\

\noindent For certain Furstenberg transformations on $\Bbb{T}^n$ we have the following theorem.

\begin{theo}\label{Furstenberg}
\textnormal{(\cite{hF61}, 2.3)}
If $\theta$ is irrational, then $\varphi_{\theta,{\boldsymbol f}}$ defines a minimal dynamical
system on $\Bbb{T}^n$. If in addition, each $f_i$ satisfies a uniform
Lipschitz condition in $\zeta_i$ for $i=1,\ldots,n-1$, then $\varphi_{\theta,{\boldsymbol f}}$ is a uniquely ergodic
transformation and the unique invariant measure is the normalized Lebesgue measure on $\Bbb{T}^n$. In particular, every affine Furstenberg transformation defines a minimal and uniquely ergodic dynamical system if $\theta$ is 
irrational.
\end{theo}

\noindent As a conclusion, we have the following result for the Furstenberg
transformation group $C^*$-algebra
${\mathscr F}_{\theta,{\boldsymbol f}}:=\mathcal{C}(\Bbb{T}^n)\rtimes_{\varphi_{\theta,{\boldsymbol f}}}\mathbb{Z}$
as introduced in \cite{rJ86}.

\begin{coro}\label{coro:1}
${\mathscr F}_{\theta,{\boldsymbol f}}=\mathcal{C}(\Bbb{T}^n)\rtimes_{\varphi_{\theta,{\boldsymbol f}}}\mathbb{Z}$ is 
a simple $C^*$-algebra for irrational $\theta$. If in addition, each $f_i$ satisfies a uniform
Lipschitz condition in $\zeta_i$ for $i=1,\ldots,n-1$, then ${\mathscr F}_{\theta,{\boldsymbol f}}$ has 
a unique tracial state.
\end{coro}
\begin{proof}
For the first part, the minimality of the action as stated in
the preceding theorem implies the simplicity of ${\mathscr F}_{\theta,{\boldsymbol f}}$
\cite{eE67,scP78}. For the second part, one can
easily check that since $\theta$ is irrational, the action of
$\mathbb{Z}$ on $\Bbb{T}^n$ generated by $\varphi_{\theta,{\boldsymbol f}}$ is
free. So there are no periodic points in $\Bbb{T}^n$. This and
the unique ergodicity of $\varphi_{\theta,{\boldsymbol f}}$ yield the result
\cite[Corollary 3.3.10, p. 91]{jT87}.
\end{proof}
\begin{rema}\label{rema:1} Using the preceding corollary and 
much like the
proof of Theorem 2.1 in \cite{kR06}, one can prove that for irrational $\theta$, 
${\mathscr F}_{\theta,{\boldsymbol f}}$ is
in fact the unique $C^*$-algebra generated by unitaries $U,V_1,\ldots,V_n$ 
satisfying
the commutator relations
\begin{equation*}
[U,V_1]=e^{2 \pi i \theta},[U,V_2]=f_1(V_1),\ldots,
[U,V_n]=f_{n-1}(V_1,\ldots,V_{n-1}) \tag*{${{\bahram}_{\boldsymbol{f}}}$},
\end{equation*}
where $[a,b]:=aba^{-1}b^{-1}$ and all other pairs of operators
from $U,V_1,\ldots,V_n$ commute. 
\end{rema}

\

\begin{rema}\label{rema:2} In \cite[Proposition 2.17]{rJ86}, R. 
Ji claims to have proved\\

\noindent \textnormal{($\ast$)} \emph{If $\varphi_{\theta,{\boldsymbol f}}$ is an ascending Furstenberg 
transformation
on $\Bbb{T}^n$ with the ascending sequence $\{k_1,k_2,\ldots,k_{n-1}\}$, 
then the
torsion subgroup of $K_*({\mathscr F}_{\theta,{\boldsymbol f}})$ is isomorphic to
$\mathbb{Z}_{k_1}\oplus\mathbb{Z}_{k_2}^{(m_2)}\oplus\ldots\oplus
\mathbb{Z}_{k_{n-1}}^{(m_{n-1})}$, where the group 
$\mathbb{Z}_{k_i}^{(m_i)}$ is the direct
product of $m_i$ copies of the cyclic group $\mathbb{Z}_{k_i}=\mathbb{Z}/k_i 
\mathbb{Z}$.}\\

\noindent From this claim one would immediately deduce that the $K$-groups 
of the \text{$C^*$-algebra} $\mathscr{A}_{n,\theta}:=\mathcal{C}(\Bbb{T}^n)\rtimes_\sigma \mathbb{Z}$
generated by an Anzai transformation $\sigma$ on $\Bbb{T}^n$ should be 
torsion-free. However, we will show that this is not true in general. This type of example first appears for $n = 6$, which seems already beyond hand calculation. (We admit that hand calculation would be the most convincing 
method to use; however, it is not practicable.) As the first 
counterexample we obtained by computer, we will see in Example \ref{counter} that $K_1(\mathscr{A}_{6,\theta})\cong\mathbb{Z}^{13}\oplus\mathbb{Z}_2$ (also, see Example \ref{compu}).
In fact, the error in the proof of ($\ast$) is in \cite[p. 29, l.2]{rJ86};
there it is ``clearly'' assumed that using a matrix ${\sf  S}$ in 
$\mathrm{GL}(2^n,\Bbb{Z})$, one can delete all entries
denoted by $\star$'s in $\mathsf{K_*-I}$, where $\mathsf{K_*}$ is the 
$2^n \times 2^n$ 
integer matrix corresponding to ${\mathscr F}_{\theta,{\boldsymbol f}}$ that acts 
on $K_*(\mathcal{C}(\Bbb{T}^n))=\Lambda^*\Bbb{Z}^n$ with respect to
a certain ordered basis. This error arose originally from the general 
form of the matrix $\mathsf{K}_*$ in \cite[p. 27]{rJ86}, which is not 
correct.
R. Ji went on to use the torsion subgroup in ($\ast$) as an invariant
to classify the $C^*$-algebras generated by ascending transformations and 
matrix algebras over them \cite[Theorem 3.6]{rJ86}. We do not know whether those
classifications holds.\\

\noindent \textbf{Question}. Do there exist two different ascending
Furstenberg transformations with the same parameter $\theta$ and 
with isomorphic transformation group $C^*$-algebras?
\end{rema}

\noindent It is worth mentioning that explicit calculations of $K$-groups in terms of the parameters involved 
are possible in low dimensions (of the tori), and answer the question raised above negatively.
However, such calculations in terms of the given integer parameters (exponents) of the ascending
Furstenberg transformation become quickly cumbersome and impossible in 
higher dimensions. We have used the computer codes in Appendix \ref{codes}
for several numerical values of the parameters in higher dimensions, and we have not
found any examples leading to the negative answer to this question yet. The torsion subgroups
of the $K$-groups are usually larger than what R. Ji claimed in \cite[Proposition 2.17]{rJ86};
it seems likely that the torsion subgroups depend on the parameters involved in such a way that Ji's 
classification result is still true. We are currently investigating this problem.\\

\noindent This paper is organized as follows. In Section 2, we review the general approach 
of exterior algebras for finding $K$-groups of transformation group $C^*$-algebras of 
homeomorphisms of the tori. In Section 3, we apply this method to the important case
of Anzai transformations and give the $K$-groups of their transformation group $C^*$-algebras
based on the tori of dimension up to 12 in Table 1 using the computer programs given in Appendix \ref{codes}.
In Section 4, we establish a Poincar\'{e} type of duality for the cokernels of integer matrices that leads
to some interesting facts about the $K$-groups when the dimension of the underlying torus is odd. In Section 5, we focus on the rank $a_n$ of the $K$-groups of Anzai transformations group $C^*$-algebras based on the 
$n$-torus, and we show that $a_n$ is, in fact, the rank of the $K$-groups of a large class of 
transformation group $C^*$-algebras including those associated with Furstenberg transformations
on the $n$-torus. In Section 6, we first uncover an interesting connection between studying $a_n$
and the irreducible representations of the Lie algebra $\mathfrak{sl}(2,\Bbb{C})$. This leads to a formula
for $a_n$ in terms of certain partitions of integers. Then we use this formula to 
show several interesting combinatorial properties of the sequence $\{a_n\}$. In Section 7,
we study the order structure of the $K_0$-group of a class of simple Furstenberg transformation
group $C^*$-algebras to make their Elliott invariants more accessible. The appendices
at the end are provided mainly for self-containment of the paper, but we 
sometimes refer to them for the proof of some lemmas or propositions that are somewhat far from the 
main concepts and goals of this paper by nature. Appendix \ref{cocompact}
contains some applications of the results of this paper to our earlier work \cite{kR06}.

\section{$K$-groups of $\mathcal{C}(\Bbb{T}^n)\rtimes_\alpha\Bbb{Z}$}\label{sec:general}

 \noindent  In this section, we describe a general method to compute the $K$-groups of 
$\mathcal{C}(\Bbb{T}^n)\rtimes_\alpha\Bbb{Z}$,
  where $\alpha$ is an arbitrary homeomorphism of $\Bbb{T}^n$. (By abuse of notation, the automorphism
  $\alpha$ of $\mathcal{C}(\Bbb{T}^n)$ is defined by $\alpha(f)=f\circ\alpha^{-1}$ for $f\in\mathcal{C}(\Bbb{T}^n)$.)
  To do this, we will pay special attention to the algebraic structure of 
$K^{*}(\Bbb{T}^n)$ and how the induced automorphisms on it can be realized. Note that it
is sufficient to consider the special case of ``linear" homeomorphisms 
since, as stated before Definition \ref{furst-anzai}, every continuous function $f:\Bbb{T}^n\to\Bbb{T}$ is homotopic to a unique ``linear" function $(\zeta_1,\ldots,\zeta_n)\mapsto\zeta_1^{b_1}\ldots \zeta_n^{b_n}$ for some integer exponents $b_1,\ldots,b_n$. Moreover, the $K$-groups of $\mathcal{C}(\Bbb{T}^n)\rtimes_\alpha\Bbb{Z}$ depend (up to isomorphism) only on the homotopy class of $\alpha$ \cite[Corollary 10.5.2]{bB98}. \\

\noindent  It is well known that $K^{*}(\Bbb{T}^n)$
  is a $\mathbb{Z}_2$-graded ring, and by the K\"{u}nneth
  formula (see \cite[Corollary 2.7.15]{mA67} or \cite[Theorem 4.1]{cS82}), 
  it is an exterior algebra (over $\mathbb{Z}$) on $n$ generators, where the
  elements of even degree are in $K^{0}(\Bbb{T}^n)$ and those of odd degree 
are in $K^{1}(\Bbb{T}^n)$.
  The generators of this exterior algebra correspond to the generators of 
the dual group $\mathbb{Z}^n$ of
  $\Bbb{T}^n$ \cite[p. 185]{jlT75}. Indeed, in this case the Chern character 
$$\mathrm{ch}:K^{*}(\Bbb{T}^n)\longrightarrow
  \Check{H}^{*}(\Bbb{T}^n,\mathbb{Q})$$ is integral and gives the Chern 
isomorphisms
  $$\mathrm{ch}_{0}:K^{0}(\Bbb{T}^n)\longrightarrow 
\Check{H}^{\text{even}}(\Bbb{T}^n,\mathbb{Z}),$$
  $$\mathrm{ch}_{1}:K^{1}(\Bbb{T}^n)\longrightarrow 
\Check{H}^{\text{odd}}(\Bbb{T}^n,\mathbb{Z}),$$
  where 
$\Check{H}^{*}(\Bbb{T}^n,\mathbb{Z})\cong
\Lambda^{*}_{\mathbb{Z}}(e_1,\ldots,e_n)$
is the (\v{C}ech) cohomology ring of $\Bbb{T}^n$ under the cup product, and
$\Check{H}^{k}(\Bbb{T}^n,\mathbb{Z})\cong
\Lambda^{k}_{\mathbb{Z}}(e_1,\ldots,e_n)$. 
On the other hand, $K^{*}(\Bbb{T}^n)\cong
K_{*}(\mathcal{C}(\Bbb{T}^n))$. So by introducing $e_i:=[z_i]_1$,
i.e. the class in $K_{1}(\mathcal{C}(\Bbb{T}^n))$
of the coordinate function $z_i:\Bbb{T}^n \rightarrow\Bbb{T}$ given by 
$$z_i(\zeta_1,\ldots,\zeta_n)=\zeta_i$$ as a unitary element of of
$\mathcal{C}(\Bbb{T}^n)$ for $i=1,\ldots,n$, we have the 
isomorphisms
$K_{*}(\mathcal{C}(\Bbb{T}^n))\cong
\Lambda^{*}_{\mathbb{Z}}(e_1,\ldots,e_n)\cong\Lambda^{*}\mathbb{Z}^n$, which respect
 the canonical embedding of $\mathbb{Z}^n$. 
Moreover, these isomorphisms are
unique since only the identity automorphism of the ring 
$\Lambda^{*}\mathbb{Z}^n$ fixes each
element of $\mathbb{Z}^n$.\\

\noindent Now, we use the Pimsner-Voiculescu six term exact sequence \cite{mP80} as the 
main tool for computing the $K$-groups
of $\mathcal{C}(\Bbb{T}^n)\rtimes_\alpha\Bbb{Z}$. Let 
$\alpha_*(=K_*(\alpha))$ be the
  ring automorphism of $K_{*}(\mathcal{C}(\Bbb{T}^n))$ induced by $\alpha$ 
and let $\alpha_i$ be the restriction
  of $\alpha_*$ on $K_{i}(\mathcal{C}(\Bbb{T}^n))$ for $i=0,1$, and
  set $A:=\mathcal{C}(\Bbb{T}^n)\rtimes_\alpha\Bbb{Z}$.
  Then we have the following exact sequence.  
\begin{equation}\label{P-V}
\begin{CD}
      K_{0}(\mathcal{C}(\Bbb{T}^n)) 
@>\alpha_{0}-\mathrm{id}>>K_{0}(\mathcal{C}(\Bbb{T}^n)) @>\jmath_{0}>>K_0(A)\\
      @A {\rm exp} AA  @. @VV \partial V \\
      K_{1}(A) @<\jmath_{1}<< K_{1}(\mathcal{C}(\Bbb{T}^n)) 
@<\alpha_{1}-\mathrm{id}<<K_1(\mathcal{C}(\Bbb{T}^n))
      \end{CD}
  \end{equation}  
  ~\\ 
    \noindent Here, $\jmath:\mathcal{C}(\Bbb{T}^n)\rightarrow A$ is the canonical 
embedding of
     $\mathcal{C}(\Bbb{T}^n)$ in $A$, $\jmath_0:=K_0(\jmath)$ and 
$\jmath_1:=K_1(\jmath)$. Also, from now on
     $\mathrm{id}$ denotes the identity function on each underlying set.
      As a result, we have the following short exact sequences
       \begin{equation}\label{K_0}
       0 \longrightarrow \mathrm{coker}
      (\alpha_0-\mathrm{id})\longrightarrow 
K_0(\mathcal{C}(\Bbb{T}^n)\rtimes_\alpha\Bbb{Z})\longrightarrow
      \ker(\alpha_{1}-\mathrm{id})\longrightarrow 0,
      \end{equation}
      \begin{equation}\label{K_1}
       0 \longrightarrow \mathrm{coker}
      (\alpha_1-\mathrm{id})\longrightarrow 
K_1(\mathcal{C}(\Bbb{T}^n)\rtimes_\alpha\Bbb{Z})\longrightarrow
      \ker(\alpha_{0}-\mathrm{id})\longrightarrow 0.
      \end{equation}~\\
      Since all the groups involved are abelian and finitely generated, and $\ker(\alpha_{i}-\mathrm{id})$ is 
torsion-free $(i=0,1)$,
      these short exact sequences split (since projective $\Bbb Z$-modules are
      precisely free abelian groups), and we have
      \begin{equation}\label{exact:K_0}
      \boxed{K_{0}(\mathcal{C}(\Bbb{T}^n)\rtimes_\alpha\Bbb{Z})\cong
      \mathrm{coker}(\alpha_{0}-\mathrm{id})\oplus \ker(\alpha_{1}-\mathrm{id}),}
      \end{equation}
      \begin{equation}\label{exact:K_1}
      \boxed{K_{1}(\mathcal{C}(\Bbb{T}^n)\rtimes_\alpha\Bbb{Z})\cong
      \mathrm{coker}(\alpha_{1}-\mathrm{id})\oplus \ker(\alpha_{0}-\mathrm{id}).}
      \end{equation}~\\
      So it suffices to determine the kernel and cokernel of 
$(\alpha_0-\mathrm{id})$ and $(\alpha_1-\mathrm{id})$ acting as
      endomorphisms on the finitely generated abelian groups 
      $\Lambda^{\text{even}}_{\mathbb{Z}}(e_1,\ldots,e_n)
      \cong\Bbb{Z}^{2^{n-1}}$
      and $\Lambda^{\text 
{odd}}_{\mathbb{Z}}(e_1,\ldots,e_n)\cong\Bbb{Z}^{2^{n-1}}$, respectively.
      Note that from the isomorphisms (\ref{exact:K_0}) and (\ref{exact:K_1}), the $K$-groups of 
$\mathcal{C}(\Bbb{T}^n)\rtimes_\alpha\Bbb{Z}$ are finitely generated abelian 
groups.
      Now, since $\alpha_*$ is a ring homomorphism, it suffices
      to know the action of $\alpha_*$ on $e_1,\ldots,e_n$. In fact, for a 
general
      basis element $e_{i_1}\wedge e_{i_2}\wedge\ldots\wedge e_{i_r}$ of 
$K_{*}(\mathcal{C}(\Bbb{T}^n))\cong
      \Lambda^{*}_{\mathbb{Z}}(e_1,\ldots,e_n )$ we have
      $$\alpha_*(e_{i_1}\wedge e_{i_2}\wedge\ldots\wedge
      e_{i_r})=\alpha_*(e_{i_1})\wedge\alpha_* 
(e_{i_2})\wedge\ldots\wedge\alpha_*(e_{i_r}).$$
      Thus if we consider $\{e_1,\ldots,e_n \}$ as the canonical basis of 
$\Bbb{Z}^n$ and
      take $\hat{\alpha}=\left.\alpha_*\right|_{\Bbb{Z}^n}$, we have 
$\alpha_*=\wedge^*\hat{\alpha}
      =\oplus_{r=1}^n \wedge^r \hat{\alpha}$,
      $\alpha_0=\wedge^{\text{even}}\hat{\alpha}=\oplus_{r \ge 0}
\wedge^{2r}\hat{\alpha}$ and
$\alpha_1=\wedge^{\text{odd}}\hat{\alpha}=\oplus_{r \ge 0}
\wedge^{2r+1}\hat{\alpha}$, where $\wedge^{i}\hat{\alpha}$ is the $i$-th 
exterior
power of $\hat{\alpha}$, which acts on 
$\Lambda^{i}{\mathbb{Z}^n}$ for $i=0,1,\ldots,n$.
       Now, let
      $\alpha^{-1}=(f_1,\ldots,f_n)$ and $a_{ji}:=A_j[f_i]$, or in other words, 
assume that
      $f_i$ is homotopic to $z_1^{a_{1i}}\ldots z_n^{a_{ni}}:(\zeta_1,\ldots,\zeta_n)\mapsto \zeta_1^{a_{1i}} 
\ldots \zeta_n^{a_{ni}}$ for $i=1,\ldots,n$. 
      So we can write
      \begin{align*}
      \alpha_*(e_i)=\alpha_*[z_i]_1=[\alpha(z_i)]_1&=[z_i\circ\alpha^{-1}]_1
      =[f_i]_1 =
      [z_1^{a_{1i}}\ldots z_n^{a_{ni}}]_1
      =\sum_{j=1}^n a_{ji}[z_j]_1=\sum_{j=1}^n a_{ji}e_j.
      \end{align*}
    Therefore $\hat{\alpha}$
      acts on $\Bbb{Z}^n$ via the corresponding integer matrix
      $\mathsf{A}:=[a_{ij}]_{n \times n}\in\mathrm{GL}(n,\Bbb{Z})$,
      $\alpha_*$ acts on $\Lambda^*\Bbb{Z}^n$ via $\wedge^*\mathsf{A}$, 
and we have the following isomorphisms
\begin{align*}
K_0(\mathcal{C}(\Bbb{T}^n)\rtimes_\alpha\Bbb{Z})&\cong 
\mathrm{coker}(\alpha_0-\mathrm{id})\oplus \ker(\alpha_1-\mathrm{id})
                            =\mathrm{coker}(\oplus_{r \ge 
0}\wedge^{2r}\hat{\alpha}-\mathrm{id})
\oplus \ker(\oplus_{r \ge 0}\wedge^{2r+1}\hat{\alpha}-\mathrm{id}),
\end{align*}
so we can write
$$\boxed{K_0(\mathcal{C}(\Bbb{T}^n)\rtimes_\alpha\Bbb{Z})
\cong \bigoplus_{r \ge 0}[\mathrm{coker}(\wedge^{2r}\hat{\alpha}-\mathrm{id})\oplus
\ker(\wedge^{2r+1}\hat{\alpha}-\mathrm{id})],}$$
and similarly
\begin{align*}
\boxed{K_1(\mathcal{C}(\Bbb{T}^n)\rtimes_\alpha\Bbb{Z})
\cong \bigoplus_{r \ge 0}[\mathrm{coker}(\wedge^{2r+1}\hat{\alpha}-\mathrm{id})
\oplus\ker(\wedge^{2r}\hat{\alpha}-\mathrm{id})].}
\end{align*}
~\\

\noindent We summarize the arguments discussed above in the following theorem.
\begin{theo}\label{theo:1}
Let $\alpha$ be a homeomorphism of $\Bbb{T}^n$ and 
$\hat{\alpha}\in\mathrm{Aut}(\Bbb{Z}^n)$
be the restriction of $\alpha_*$ to $\Bbb{Z}^n $ (as above). Then
$\alpha_*=\wedge^*\hat{\alpha}=\oplus_{r=1}^n \wedge^r \hat{\alpha}$
on $K^*(\Bbb{T}^n)=\Lambda^*\Bbb{Z}^n$ and\\
\begin{align*}
K_0(\mathcal{C}(\Bbb{T}^n)\rtimes_\alpha\Bbb{Z})
&\cong \bigoplus_{r \ge 0}[\mathrm{coker}(\wedge^{2r}\hat{\alpha}-\mathrm{id})\oplus
\ker(\wedge^{2r+1}\hat{\alpha}-\mathrm{id})],\\
K_1(\mathcal{C}(\Bbb{T}^n)\rtimes_\alpha\Bbb{Z})&\cong \bigoplus_{r \ge 
0}[\mathrm{coker}(\wedge^{2r+1}\hat{\alpha}-\mathrm{id})\oplus
\ker(\wedge^{2r}\hat{\alpha}-\mathrm{id})].
\end{align*}
\end{theo}
~\\

\noindent Therefore in order to compute the $K$-groups of
$\mathcal{C}(\Bbb{T}^n)\rtimes_\alpha\Bbb{Z}$, we must find the kernel and 
cokernel of $\wedge^{r}\hat{\alpha}-\mathrm{id}$ as an endomorphism of $\Lambda^{r}{\mathbb{Z}^n}$
for $r=0,1,\ldots,n$. Note that the matrix of $\wedge^{r}\hat{\alpha}-\mathrm{id}$ with 
respect to the canonical basis
$\{e_{i_1}\wedge\ldots\wedge e_{i_r}|1 \le i_1 <  \ldots < i_r \le n \}$ 
with lexicographic
order is $\mathsf{A}_{n,r}:=\wedge^r \mathsf{A}- \mathsf{I}_{\binom{n}{r}}$, 
which is an integer
matrix of order $\binom{n}{r}$ ($\mathsf{I}_k$ is the identity matrix of 
order $k$ - we often
omit $k$ whenever it is clear). So by computing
the kernel and cokernel of $\mathsf{A}_{n,r}$ for $r=0,1,\ldots,n$ with 
appropriate tools (such as the Smith normal form), one can 
determine the $K$-groups of 
$\mathcal{C}(\Bbb{T}^n)\rtimes_\alpha\Bbb{Z}$. 
The author has written some Maple codes to handle such computations
(see Appendix \ref{codes}) .

\begin{coro}\label{coro:3}
The $K$-groups of $\mathcal{C}(\Bbb{T}^n)\rtimes_\alpha\Bbb{Z}$ are finitely 
generated abelian groups with the same rank. Moreover, this common rank equals
$$\mathrm{rank}\ker(\wedge^* \hat{\alpha}-\mathrm{id})=
\sum_{r=0}^n \mathrm{rank}\ker(\wedge^r \hat{\alpha}-\mathrm{id}).$$
\end{coro}
\begin{proof}
Use the previous proposition and note that for any $\varphi\in\mathrm{End}(\Bbb{Z}^n)$ one has 
$\mathrm{rank}\ker\varphi=\mathrm{rank}\hspace{2pt}\mathrm{coker}\hspace{2pt}\varphi$
by the Smith normal form theorem (see Theorem \ref{theo:smith}).
\end{proof}
\begin{coro}
If $\alpha, \beta$ are homeomorphisms of $\Bbb T^n$, whose corresponding 
integer matrices $\mathsf{A, B}\in\mathrm{GL}(n,\Bbb Z)$ are similar over $\Bbb Z$, then 
$$K_j(\mathcal{C}(\Bbb{T}^n)\rtimes_\alpha\Bbb{Z})\cong K_j(\mathcal{C}(\Bbb{T}^n)\rtimes_\beta\Bbb{Z}),
\hspace{.3in} (j=1,2).$$
\end{coro}
\begin{proof}
The assumption obviously implies that the automorphisms $\hat\alpha$ and $\hat\beta$ are conjugate in 
$\mathrm{Aut}(\Bbb Z^n)$. This together with an easy application of the identity $\wedge^r(\hat\phi\circ\hat\psi)=(\wedge^r\hat\phi)\circ(\wedge^r\hat\psi)$ (see Appendix \ref{endomorphism}) imply that
$\wedge^r\hat\alpha$ and $\wedge^r\hat\beta$ (and therefore $\wedge^r\hat\alpha-\mathrm{id}$ and $\wedge^r\hat\beta-\mathrm{id}$)
are conjugate in $\mathrm{Aut}(\Lambda^r\Bbb Z^n)$ for $r=0,1,\ldots,n$. The result follows now from Theorem \ref{theo:1}.
\end{proof}

\section{Anzai transformation group $C^*$-algebras $\mathscr{A}_{n,\theta}$}\label{sec:Anzai}

\noindent The simplest case of a Furstenberg transformation on an $n$-torus is an Anzai transformation $\sigma$, which 
was defined in part (d) of Definition \ref{furst-anzai}. To study the $K$-groups of Anzai transformation group $C^*$-algebras $\mathscr{A}_{n,\theta}=
\mathcal{C}(\Bbb{T}^n)\rtimes_\sigma\Bbb{Z}$ using methods of the previous section, we will first need the ``linearized" form of the corresponding affine homeomorphism $\sigma^{-1}$, which is as follows
$$(\zeta_1,\zeta_2,\ldots,\zeta_n)\mapsto(\zeta_1,\zeta_1 \zeta_2,\ldots,\zeta_{n-1}\zeta_n).$$
      So $\hat{\sigma}(e_i)=e_{i-1}+e_i$ for $i=1,\ldots,n$ $(e_0:=0)$.
    The matrix  with respect to the canonical basis $\{e_1,\ldots,e_n \}$ of
    $\mathbb{Z}^n$ that corresponds to $\hat{\sigma}$ is the full Jordan block

      \[\mathsf{S}_n := \left(\begin{array}{ccccc}
1 & 1 & 0 & \cdots& 0\\
0 & 1 & 1 & &\vdots\\
0 &\ddots&\ddots&\ddots&0\\
\vdots&   & 0 & 1 & 1\\
0 &\cdots  & 0 & 0 & 1
\end{array}
\right)_{n \times n}
\]\\

\noindent The following examples illustrates the methods described in the previous section 
for computing the $K$-groups of Anzai transformation group $C^*$-algebras $\mathscr{A}_{n,\theta}$.

\begin{example}
We compute the $K$-groups of $\mathscr{A}_{3,\theta}$, which were computed in 
\cite{sW02} by another method (the $C^*$-algebra was denoted by $A^{5,5}_\theta$ in there).
In fact, the Chern character and noncommutative geometry were used in 
\cite{sW02} to compute the
kernel and cokernel of $\sigma_i -\mathrm{id}$
for $i=0,1$. However, we compute the kernel 
and cokernel of $\mathsf{S}_{3,r}:=\wedge^r \mathsf{S}_3-
\mathsf{I}_{\binom{3}{r}}$ for $r=0,1,2,3$, where
\[\mathsf{S}_3 = \left(\begin{array}{ccccc}
1 & 1 & 0 \\
0 & 1 & 1 \\
0 & 0 & 1\\
\end{array}
\right)
\]\\

\noindent $\boldsymbol{r=0})$ $\mathsf{S}_{3,0}=\wedge^0 \mathsf{S}_3-\mathsf{I}_1=[0]$. So
$\ker \mathsf{S}_{3,0}=\mathbb{Z}$ and
$\mathrm{coker}\hspace{2pt} \mathsf{S}_{3,0}=\mathbb{Z}/\langle 0 
\rangle\cong
\mathbb{Z}.$\\

 \noindent $\boldsymbol{r=1})$ 
 $$\mathsf{S}_{3,1}=\wedge^1 
\mathsf{S}_3-\mathsf{I}_3=\left(
\begin{array}{ccc}
1 & 1 & 0 \\
0 & 1 & 1 \\
0 & 0 & 1
\end{array}
\right)
-\left(\begin{array}{ccc}
1 & 0 & 0 \\
0 & 1 & 0 \\
0 & 0 & 1\\
\end{array}
\right)=\left(\begin{array}{ccc}
0 & 1 & 0 \\
0 & 0 & 1 \\
0 & 0 & 0\\
\end{array}
\right)$$
So
\begin{align*}
\ker \mathsf{S}_{3,1}&=\{(x,y,z)\in\mathbb{Z}^3|\,y=z=0\}=(\mathbb{Z},0,0)\cong\mathbb{Z},\\
\mathrm{coker}\hspace{2pt}\mathsf{S}_{3,1}&=
\mathbb{Z}^3/\mathsf{S}_{3,1}\mathbb{Z}^3
=\mathbb{Z}^3/\langle e_1,e_2 \rangle \cong \mathbb{Z}.
\end{align*}

\noindent $\boldsymbol{r=2})$ 
$$\mathsf{S}_{3,2}=\wedge^2\mathsf{S}_3-\mathsf{I}_3=
\left(\begin{array}{ccc}
1 & 1 & 1 \\
0 & 1 & 1 \\
0 & 0 & 1\\
\end{array}
\right)-\left(\begin{array}{ccc}
1 & 0 & 0 \\
0 & 1 & 0 \\
0 & 0 & 1\\
\end{array}
\right)=\left(\begin{array}{ccc}
0 & 1 & 1 \\
0 & 0 & 1 \\
0 & 0 & 0 \\
\end{array}
\right)$$
So
\begin{align*}
\ker \mathsf{S}_{3,2}&=\{(x,y,z)\in\mathbb{Z}^3|\,y+z=z=0 
\}=(\mathbb{Z},0,0)\cong\mathbb{Z},\\
\mathrm{coker}\hspace{2pt}\mathsf{S}_{3,2}&=
\mathbb{Z}^3/\mathsf{S}_{3,2}\mathbb{Z}^3
=\mathbb{Z}^3/\langle e_1,e_2 \rangle \cong \mathbb{Z}.
\end{align*}
\noindent $\boldsymbol{r=3})$ $\mathsf{S}_{3,3}=\wedge^3 
\mathsf{S}_3-\mathsf{I}_1=[0]$. So
$\ker \mathsf{S}_{3,3}
=\mathbb{Z}$ and $\mathrm{coker}\hspace{2pt} 
\mathsf{S}_{3,3}=\mathbb{Z}/\langle 0 \rangle\cong
\mathbb{Z}.$ \\

\noindent Now, using Theorem \ref{theo:1} we have
\begin{align*}
K_0(A^{5,5}_\theta)&=K_0({\mathscr{A}_{3,\theta}})\cong 
(\mathrm{coker}\hspace{2pt}
\mathsf{S}_{3,0}\oplus \mathrm{coker}\hspace{2pt}\mathsf{S}_{3,2})\oplus
(\ker \mathsf{S}_{3,1}\oplus \ker \mathsf{S}_{3,3})
\cong\mathbb{Z}\oplus\mathbb{Z}\oplus\mathbb{Z}\oplus\mathbb{Z}=\mathbb{Z}^4,\\
K_1(A^{5,5}_\theta)&=K_1({\mathscr{A}_{3,\theta}})\cong (\mathrm{coker}\hspace{2pt}\mathsf{S}_{3,1}\oplus 
\mathrm{coker}\hspace{2pt}\mathsf{S}_{3,3})\oplus(\ker \mathsf{S}_{3,0}\oplus\ker\mathsf{S}_{3,2})\cong\mathbb{Z}\oplus\mathbb{Z}
\oplus\mathbb{Z}\oplus\mathbb{Z}=\mathbb{Z}^4.
\end{align*}
\end{example}
~\\
%\vspace{.2in}

%\begin{sidewaystable}[]
%\begin{center}
%\vspace{5in}
%\caption{\Large\it The $K$-groups of $\mathscr{A}_{n,\theta}$ for $1 \leq n \leq 
%12$}\vspace{4pt}
%{\LARGE
%\begin{tabular}{|c|c|c|c|}
%%%%%%%%Table%%%%%%%%%
%\end{tabular}
%}
%\end{center}
%\end{sidewaystable}
%\newpage
\begin{nota}\label{not:a_n}
We let $a_n:=\textnormal{rank}\hspace{2pt}K_0(\mathscr{A}_{n,\theta})=
\textnormal{rank}\hspace{2pt}K_1(\mathscr{A}_{n,\theta})$ and $a_{n,r}:=
\textnormal{rank}\ker(\wedge^r \mathsf{S}_n-\mathsf{I})$ for 
$r=0,1,\ldots,n$. From Corollary \ref{coro:3}
we have $$a_n=\textnormal{rank}\ker(\wedge^* 
\mathsf{S}_n-\mathsf{I})=\sum_{r=0}^n a_{n,r}.$$
\end{nota}\label{nota:2}

\begin{example}\label{compu}
Using the methods described in Section \ref{sec:general}, we have obtained the $K$-groups of
$\mathscr{A}_{n,\theta}$ by computer for $1 \leq n \leq 12$. The cases $n=1, 2, 3$ have been calculated in the 
literature already:  $\mathscr{A}_{1,\theta}=A_\theta$ in \cite{mR81}; $\mathscr{A}_{2,\theta}=A^4_\theta$ in
\cite{pM93}; and $\mathscr{A}_{3,\theta}=A^{5,5}_\theta$ in \cite{sW02}. However, there are no explicit computations for 
the higher dimensional cases starting with $\mathscr{A}_{4,\theta}=A^{6,10}_\theta$ as in \cite{pM02} since 
hand calculations of kernels and cokernels of the maps become quickly impossible. Using the Maple codes given 
in Appendix \ref{codes}, we can find the 
kernel and cokernel of $$\mathsf{S}_{n,r}:=\wedge^r \mathsf{S}_n-\mathsf{I}_{\binom{n}{r}}$$ for $r=0,1,\ldots,n$ and $n=1,\ldots,12$ by means of the Smith normal form theorem (see Appendix \ref{smith}), and therefore we can compute the $K$-groups. The results are illustrated
in Table 1, where $\mathbb{Z}_{k}^{(m)}$ denotes the direct sum of $m$ 
copies of the cyclic group $\mathbb{Z}_{k}=\mathbb{Z}/k \mathbb{Z}$. \\
\begin{table}[ht]
\caption{\it The $K$-groups of $\mathscr{A}_{n,\theta}$ for $1 \leq n \leq 12$}
\begin{tabular}{|c|c|c|c|}
\hline
$n$  & $K_0(\mathscr{A}_{n,\theta})$ & $K_1(\mathscr{A}_{n,\theta})$ & $a_n$ 
\\ \hline\hline
1 & $\Bbb{Z}^2$ & $\Bbb{Z}^2$ & 2 \\  \hline
2 & $\Bbb{Z}^3$ & $\Bbb{Z}^3$ & 3  \\   \hline
3 & $\Bbb{Z}^4$ & $\Bbb{Z}^4$ & 4  \\   \hline
4 & $\Bbb{Z}^6$ & $\Bbb{Z}^6$ & 6  \\   \hline
5 & $\Bbb{Z}^8$ & $\Bbb{Z}^8$ & 8  \\   \hline
6 & $\Bbb{Z}^{13}$ & $\Bbb{Z}^{13} \oplus \Bbb{Z}_2$ & 13  \\   \hline
7 & $\Bbb{Z}^{20}$ & $\Bbb{Z}^{20}$ & 20  \\   \hline
8 & $\Bbb{Z}^{32} \oplus \Bbb{Z}_8^{(2)} $ & $\Bbb{Z}^{32} \oplus 
\Bbb{Z}_{18}^{(2)}$ & 32  \\   \hline
9 & $\Bbb{Z}^{52} \oplus \Bbb{Z}_3^{(2)} \oplus \Bbb{Z}_9^{(2)}$ & 
$\Bbb{Z}^{52} \oplus \Bbb{Z}_3^{(2)} \oplus \Bbb{Z}_9^{(2)}$ & 52  \\   \hline
10 & $\Bbb{Z}^{90} \oplus \Bbb{Z}_{55}^{(4)}$ & $\Bbb{Z}^{90} \oplus 
\Bbb{Z}_{11}^{(2)} \oplus \Bbb{Z}_{99} \oplus \Bbb{Z}_{198} \oplus 
\Bbb{Z}_{2574}$ & 90  \\   \hline
11 & $\Bbb{Z}^{152} \oplus \Bbb{Z}_{11}^{(12)} \oplus\Bbb{Z}_{143}^{(4)} 
\oplus \Bbb{Z}_{286}^{(2)}$ & $\Bbb{Z}^{152} \oplus \Bbb{Z}_{11}^{(12)} 
\oplus \Bbb{Z}_{143}^{(4)} \oplus \Bbb{Z}_{286}^{(2)}$ & 152  \\   \hline
12 & $\Bbb{Z}^{268} \oplus \Bbb{Z}_{13}^{(14)} \oplus\Bbb{Z}_{26}^{(4)} 
\oplus \Bbb{Z}_{1716}^{(4)}\oplus \Bbb{Z}_{3432}^{(2)}\oplus \Bbb{Z}_{58344}^{(2)}$
& $\Bbb{Z}^{268} \oplus \Bbb{Z}_{13}^{(4)} \oplus\Bbb{Z}_{26}^{(4)} 
\oplus \Bbb{Z}_{286}^{(6)}\oplus \Bbb{Z}_{4862}^{(2)}\oplus \Bbb{Z}_{68068}^{(2)}$  
& 268 
\\ \hline
\end{tabular}
\end{table}

\noindent Due to computational limitations, we do not have
any results yet for $n > 12$, except for the sequence of ranks $\{a_n\}$, which we will
study in detail in Sections \ref{sec:rank} and \ref{sec:comb}. 
We will show the importance 
of this sequence in Section \ref{sec:rank}. Briefly, $a_n$ is the common rank of the \text{$K$-groups} 
of a certain family of $C^*$-algebras including Furstenberg transformation
group $C^*$-algebras ${\mathscr F}_{\theta,{\boldsymbol f}}$ based on $\Bbb{T}^n$.
Also, we will prove that $\{a_n\}$ is a
strictly increasing sequence (see Proposition \ref{prop:2}). On the other hand,
it seems that the \text{$K$-groups} of $\mathscr{A}_{n,\theta}$ 
have torsion in general. The first example is $K_1(\mathscr{A}_{6,\theta})$; this is in fact because 
$\mathrm{coker}\hspace{2pt}\mathsf{S}_{6,3}=\mathrm{coker}(\wedge^3 \mathsf{S}_6-\mathsf{I}_{20})\cong
\Bbb{Z}^{3} \oplus \Bbb{Z}_2$ (see Example \ref{counter} and Remark \ref{rema:2}).
Also, it is seen that the $K_0$- and $K_1$-groups are isomorphic for odd values of $n$ in Table 1. 
In fact, this is true for more general cases (see Theorem \ref{odd-tori}).
\end{example}

\section{A Poincar\'{e} type of duality}

\noindent As stated in Theorem \ref{theo:1}, the $K$-groups of 
a transformation group $C^*$-algebras of the form $\mathcal{C}(\Bbb{T}^n)\rtimes_\alpha
\Bbb{Z}$ are completely determined by the corresponding homomorphism $\hat\alpha\in\mathrm
{Aut}(\Bbb{Z}^n)$ and its exterior powers. From a computational point of 
view, we only need the cokernels of the maps involved, since we know that for any 
endomorphism $\varrho$ on $\Bbb{Z}^m$, $\ker\varrho\cong\mathrm{coker}\varrho/{\mathrm{tor}}(\mathrm{coker}\varrho)$, where ${\mathrm{tor}}(G)$ denotes the torsion subgroup of the finitely generated abelian group $G$ (see Appendix \ref{smith}). When $\det\hat\alpha=1$, we don't even need to  compute
all the cokernels. This is due to the following proposition, which
 establishes a Poincar\'{e} type of duality between cokernels of certain integer matrices. We refer to Definition \ref{equiv}
 for the notion of equivalence of endomorphisms of $\Bbb Z^n$.
 
\begin{prop}[Poincar\'{e} duality]\label{poincare}
Let $\hat\alpha\in\mathrm{SL}(n,\Bbb{Z})$ (i.e. $\det\hat\alpha=1$). Then
$\wedge^r \hat\alpha-\mathrm{id}$ and $\wedge^{n-r} \hat\alpha-\mathrm{id}$ 
are equivalent as endomorphisms of $\Lambda^r 
\Bbb{Z}^n=\Lambda^{n-r}\Bbb{Z}^n=\Bbb{Z}^{\binom{n}{r}}$.
Equivalently, $\mathrm{coker}(\wedge^r \hat\alpha-\mathrm{id})$ is isomorphic to
$\mathrm{coker}(\wedge^{n-r} \hat\alpha-\mathrm{id})$ for $r=0,1,\ldots,n$.
\end{prop}
\begin{proof}
We prove the equivalence of the endomorphisms for their corresponding integer matrices with respect to
a certain basis. Let $\mathcal{E}=\{e_1,\ldots,e_n \}$
be a basis for $\Bbb{Z}^n$ and set $S=\{1,2,\ldots,n\}$. For $I 
=\{i_1,\ldots,i_r \}\subset S$ with
$1 \leq i_1<\ldots<i_r \leq n$, put $e_I=e_{i_1}\wedge\ldots\wedge 
e_{i_r}\in \Lambda^r
\Bbb{Z}^n$. Then $\mathcal{E}_r:=\{e_I \mid I \subset S~,~|I|=r \}$ is a 
basis
for $\Lambda^r \Bbb{Z}^n$. Let $\omega:=e_1 \wedge\ldots\wedge e_n$, which 
generates
$\Lambda^n \Bbb{Z}^n$. We have $\wedge^0 \hat\alpha-\mathrm{id}=0$, and
$\wedge^n 
\hat\alpha(\omega)=\hat\alpha(e_1)\wedge\ldots\wedge\hat\alpha(e_n)
=(\det\hat\alpha)(e_1 \wedge\ldots\wedge e_n)=\omega$, so $\wedge^n 
\hat\alpha-\mathrm{id}=0$.
Now, fix an $r \in\{1,\ldots,n-1\}$. For an arbitrary subset $I \subset S$ 
with
$|I|=r$, take $J=\mathcal{E}\setminus I=\{j_1,\ldots,j_{n-r} \}$, so 
$|J|=n-r$. Then $e_I \wedge e_J=(\mathrm{sgn}\hspace{2pt}\mu)\,\omega$, in which $\mu\in 
S_n$ is the permutation that converts $(1,2,\ldots,n)$ to $(i_1,\ldots,i_r,j_1,\ldots,j_{n-r})$. It 
is easily seen that $\mu=\mu_1 \ldots \mu_r$, where $\mu_k$ is the permutation that takes 
$i_k$ from its position in $(1,2,\ldots,n)$ to its new position in $(i_1,\ldots,i_r,j_1,\ldots,j_{n-r})$.
One can see that $\mu_k$ is the combination of $i_{k}-(r-k+1)$ number of transpositions ($k=1,\ldots,r$). Thus
$$\mathrm{sgn}\hspace{2pt}\mu=\prod_{k=1}^r(-1)^{i_{k}-(r-k+1)}=
(-1)^{\ell(I)-\frac{r(r+1)}{2}},$$ where $\ell(I):=\sum_{k=1}^r
i_k$. Now, take $m=\binom{n}{r}=\binom{n}{n-r}$ and let
$\mathcal{E}_r=\{e_{I_1},\ldots, e_{I_m}\}$ be a basis
for $\Lambda^r \Bbb{Z}^n$. Write $\mathcal{E}_{n-r}=
\{e_{J_1},\ldots,e_{J_m}\}$ as the basis for
$\Lambda^{n-r} \Bbb{Z}^n$ such that $J_k=\mathcal{E}\setminus
I_k$ for $k=1,\ldots,m$. From the above argument one can write
$$e_{I_i}\wedge 
e_{J_j}=(-1)^{\ell(I_i)-\frac{r(r+1)}{2}}\delta_{ij}\omega,$$
since if $i \ne j$ then $I_i \cap J_j \ne \emptyset$ and $e_{I_i}\wedge 
e_{J_j}=0$.
Let $\mathsf{A}=[a_{ij}]_{m \times m}$ and $\mathsf{B}=[b_{ij}]_{m 
\times m}$ be the corresponding
integer matrices of $\wedge^r \hat\alpha$ and $\wedge^{n-r} \hat\alpha$ with 
respect to
$\mathcal{E}_r$ and $\mathcal{E}_{n-r}$, respectively. So
$\wedge^r \hat\alpha(e_{I_i})=\sum_{p=1}^m a_{pi}e_{I_p}$ and
$\wedge^{n-r} \hat\alpha(e_{J_j})=\sum_{q=1}^m b_{qj}e_{J_q}$. What we want 
to show is that
$\mathsf{A}-\mathsf{I}$ is equivalent to $\mathsf{B}-\mathsf{I}$. We have
%\begin{multline*}
$$
\wedge^n \hat\alpha(e_{I_i}\wedge e_{J_j})
=(-1)^{\ell(I_i)-\frac{r(r+1)}{2}}\delta_{ij}\omega=
\wedge^r
\hat\alpha(e_{I_i})\bigwedge\wedge^{n-r}\hat\alpha(e_{J_j})=
\sum_{p,q=1}^m
a_{pi}b_{qj}(-1)^{\ell(I_p)-\frac{r(r+1)}{2}}\delta_{pq}\omega.
$$
%\end{multline*}
Therefore one obtains
\begin{equation}\label{CB=I}
\sum_{k=1}^m (-1)^{\ell(I_k)-\ell(I_i)}a_{ki}b_{kj}=\delta_{ij}.
\end{equation}
Now, if we set $c_{ij}:=(-1)^{\ell(I_j)-\ell(I_i)}a_{ji}$ and
$\mathsf{C}:=[c_{ij}]_{m \times m}$, then
$c_{ij}-\delta_{ij}=(-1)^{\ell(I_j)-\ell(I_i)}(a_{ji}-\delta_{ji})$.
Therefore $\mathsf{C}-\mathsf{I}$ is obtained from $\mathsf{A}-\mathsf{I}$ 
by changing rows (and columns)
and occasionally multiplying some rows (and columns) by $-1$. This means that
$\mathsf{C}-\mathsf{I}$ is equivalent to $\mathsf{A}-\mathsf{I}$.
On the other hand, the equation (\ref{CB=I}) means that ${\sf CB=I}$. So
$\mathsf{C}-\mathsf{I}=\mathsf{C}(\mathsf{B}-\mathsf{I})(-\mathsf{I})$ 
and $\mathsf{B}-\mathsf{I}$ is also equivalent to $\mathsf{C}-\mathsf{I}$. Consequently, $\mathsf{A}-\mathsf{I}$ is 
equivalent to $\mathsf{B}-\mathsf{I}$.
\end{proof}

\begin{coro}
If  $\det\hat\alpha=1$, then $\mathrm{rank}\ker(\wedge^{r}
\hat\alpha-\mathrm{id})= \mathrm{rank}\ker(\wedge^{n-r} \hat\alpha-\mathrm{id})$.
In particular, using Notation \ref{not:a_n}, we have $a_{n,r}=a_{n,n-r}$ for $r=0,1,\ldots,n$.
\end{coro}
\noindent We are now ready to apply our Poincar\'{e} duality to a the following
$K$-theoretic result.
\begin{theo}\label{odd-tori}
Let $A:=\mathcal{C}(\Bbb{T}^{2m-1})\rtimes_\alpha\Bbb{Z}$ be such that the 
corresponding homomorphism $\hat\alpha$ satisfies $\det\hat\alpha=1$. Then $K_0(A)\cong K_1(A)$ as abelian groups,
 and
the (common) rank of the $K$-groups of $A$ is an even number.
In particular, for every Furstenberg transformation group $C^*$-algebra 
${\mathscr F}_{\theta,{\boldsymbol f}}$
based on an \text{odd-dimensional} torus (e.g. $\mathscr{A}_{2m-1,\theta}$), 
one has
$K_0({\mathscr F}_{\theta,{\boldsymbol f}})\cong K_1({\mathscr F}_{\theta,{\boldsymbol f}})$.
\end{theo}
\begin{proof}
Combining Theorem \ref{theo:1} and Proposition \ref{poincare}, one 
obtains
$$K_0(A)\cong K_1(A)\cong\bigoplus_{k=0}^{m-1}[\mathrm{coker}(\wedge^k 
\hat\alpha-\mathrm{id})
\oplus \ker(\wedge^k \hat\alpha-\mathrm{id})].$$
As a result, the rank of the $K$-groups of $A$ is an even number since the 
ranks of
the cokernel and kernel of an endomorphism coincide. Note that for 
${\mathscr F}_{\theta,{\boldsymbol f}}$
the corresponding integer matrix of $\hat\alpha$ is an upper triangular 
matrix with
1's on the diagonal. Thus $\det\hat\alpha=1.$
\end{proof}

\section{The rank $a_n$ of the $K$-groups of $\mathscr{A}_{n,\theta}$}\label{sec:rank}

\noindent In this section, we study some general properties of $a_n$, the (common) rank of the $K$-groups of
Anzai transformation group $C^*$-algebras based on $\Bbb T^n$. We specify a
family of \text{$C^*$-algebras}, whose ranks of $K$-groups are given by the same sequence 
$\{a_n \}$. As an application, we characterize the rank of the $K$-groups of 
Furstenberg transformation group \text{$C^*$-algebras}
${\mathscr F}_{\theta,{\boldsymbol f}}$. In Appendix \ref{cocompact}, this study will 
have some applications to the classification of simple infinite dimensional quotients of the Heisenberg-type group $C^*$-algebras $C^*(\mathfrak{D}_n)$, which were studied in an earlier work \cite{kR06}. We remind the reader of some linear algebraic properties of nilpotent and unipotent 
matrices in Appendix \ref{nilpotent}.\\

\noindent We compare the ranks of the $K$-groups of a class of $C^*$-algebras of the form $\mathcal{C}(\Bbb{T}^n)\rtimes_\alpha\Bbb{Z}$ in the following theorem, which shows that the rank $a_n$ 
of the $K$-groups of $\mathscr{A}_{n,\theta}$ is somehow generic. 

\begin{theo}\label{rank}
Let $A=\mathcal{C}(\Bbb{T}^n)\rtimes_\alpha\Bbb{Z}$, in which $\alpha$ is a 
homeomorphism
of $\Bbb{T}^n$, whose corresponding integer matrix 
$\mathsf{A}\in\mathrm{GL}(n,\Bbb{Z})$ is unipotent of
maximal degree (i.e. $\deg(\mathsf{A})=n$). Then
$$\mathrm{rank}\hspace{2pt}K_0(A)=
\mathrm{rank}\hspace{2pt}K_1(A)=a_n
=\mathrm{rank}\hspace{2pt}K_0(\mathscr{A}_{n,\theta})
=\mathrm{rank}\hspace{2pt}K_1(\mathscr{A}_{n,\theta}).$$
In particular, the rank of the $K$-groups of any Furstenberg transformation 
group \text{$C^*$-algebra} ${\mathscr F}_{\theta,{\boldsymbol f}}=
\mathcal{C}(\Bbb{T}^n)\rtimes_{\varphi_{\theta,{\boldsymbol f}}}\Bbb{Z}$ is
equal to the rank of the $K$-groups of $\mathscr{A}_{n,\theta}$, namely, to 
$a_n$.
\end{theo}
\begin{proof}
Let $\hat{\alpha}$ denote the restriction of $\alpha_*$ to $\Bbb{Z}^n$ and 
$\mathsf{A}$ be
the corresponding matrix of $\hat\alpha$ acting on $\Bbb{Z}^n$. Also, let 
$\mathsf{S}_n$ be the corresponding matrix for $\mathscr{A}_{n,\theta}$ as denoted
in Section \ref{sec:Anzai}. Since $\mathsf{A}$ 
is unipotent of maximal degree by assumption, and $\mathsf{S}_n$ is unipotent of maximal 
degree too,  the matrices $\mathsf{A}$ and $\mathsf{S}_n$ are similar over $\Bbb C$ (see Corollary \ref{coro:similar}). In fact, the Jordan normal form of $\mathsf{A}-\mathsf{I}$ is precisely $\mathsf{S}_n-\mathsf{I}$.
On the other hand, we know by Corollary \ref{coro:3} that the rank of the 
$K$-groups of $A$ is equal to $\mathrm{rank}\ker(\wedge^*\mathsf{A}-\mathsf{I})$. Note that by the Smith normal form theorem (see Theorem \ref{theo:smith}), $\mathrm{rank}\ker(\wedge^*\mathsf{A}-\mathsf{I})=
\dim_{\Bbb C}\ker(\wedge^*\mathsf{A}-\mathsf{I})$. The similarity of $\mathsf{A}$ and $\mathsf{S}_n$ implies the similarity of $\wedge^*\mathsf{A}-\mathsf{I}$ and $\wedge^*\mathsf{S}_n-\mathsf{I}$ as matrices acting on $\Lambda^*\Bbb{C}^n$. So $\dim_{\Bbb{C}}\ker(\wedge^*\mathsf{A}-\mathsf{I})=\dim_{\Bbb{C}}\ker(\wedge^*\mathsf{S}_n-\mathsf{I})=a_n$, which yields the 
result.\\

\noindent For the second part, note that the corresponding integer matrix of a 
Furstenberg transformation $\varphi_{\theta,{\boldsymbol f}}$ on $\Bbb{T}^n$ is of the form

\[\left(\begin{array}{ccccc}
1 & b_{12} & b_{13} & \cdots& b_{1n}\\
0 & 1 & b_{23} & &\vdots\\
0 &\ddots&\ddots&\ddots& b_{n-2,n}\\ \tag{$\heartsuit$}
\vdots&   & 0 & 1 & b_{n-1,n}\\
0 &\cdots  & 0 & 0 & 1
\end{array}
\right)_{n \times n}
\]
which is unipotent of maximal degree since $b_{i,i+1}\ne 0$ for $i=1,\ldots,n-1$ (see Definition \ref{furst-anzai}
and Example \ref{furst-matrix}). Now, the proof of
the first part yields the result.
\end{proof}

\begin{rema}\label{rema:3}
In the preceding theorem, the basis for $\Bbb{Z}^n$ for the 
matrices involved is $\{e_1,\ldots,e_n \}$, where $e_i:=[z_i]_1$
as introduced at the beginning of Section \ref{sec:general}. It is interesting to know that
if $\hat{\alpha}$ is an arbitrary unipotent automorphism of $\Bbb{Z}^n$, 
then there is a basis for $\Bbb{Z}^n$ with respect to which the integer matrix $\mathsf{A}$ of 
$\hat{\alpha}$ is of the form
$(\heartsuit)$ above (but not necessarily with $b_{i,i+1}\ne 0$ for 
$i=1,\ldots,n-1$,
unless $\hat{\alpha}$ is of maximal degree) \cite[Theorems 16 and 18]{fjH63}.
The unipotency of $\hat{\alpha}$ has also important effects on the dynamics 
of the generated flow on $\Bbb T^n$. For example, if $\alpha$ is an affine transformation on $\Bbb{T}^n$ 
and $\hat{\alpha}$
is unipotent, then the dynamical system $(\Bbb{T}^n,\alpha)$ has 
quasi-discrete
spectrum \cite[Theorem 19]{fjH63}. More generally, let
$\alpha=(\boldsymbol{t},\mathsf{A})$ be an affine transformation on 
$\Bbb{T}^n$
and take $Z_p(\mathsf{A})=
\ker(\mathsf{A}^p-\mathrm{id})\subset\Bbb{Z}^n$ for $p \in \Bbb{N}$  and consider the 
following conditions

\begin{itemize}
\item [\textnormal{(1)}] $Z_1(\mathsf{A})=Z_p(\mathsf{A})$, $\forall p 
\in\Bbb{N},$
\item [\textnormal{(2)}] $\boldsymbol{t}$ is rationally independent over
$Z_1(\mathsf{A}),$ i.e. if $\boldsymbol{k}=(k_1,\ldots,k_n)\in 
Z_1(\mathsf{A})$
is such that $\langle \boldsymbol{t},\boldsymbol{k}\rangle:=\sum_{j=1}^n t_j 
k_j$ is a rational
number, then $\boldsymbol{k}=\boldsymbol 0$.
\item [\textnormal{(3)}] $Z_1(\mathsf{A})\ne \{0\},$
\item [\textnormal{(4)}] $\mathsf{A}$ is unipotent.
\end{itemize}

\noindent Then $(\Bbb{T}^n,\alpha)$ is ergodic with respect to Haar measure if and 
only if $\alpha$
satisfies the conditions (1) and (2) \cite{fjH63}. Moreover, if $\alpha$ satisfies the conditions (1) 
through (4), then the dynamical system
$(\Bbb{T}^n,\alpha)$ is minimal, uniquely ergodic with respect to Haar 
measure, and has
quasi-discrete spectrum. Conversely, any minimal transformation on 
$\Bbb{T}^n$ with
topologically quasi-discrete spectrum is conjugate to an affine 
transformation which
must satisfy the conditions (1) through (4) \cite{fjH65}. The $C^*$-algebras corresponding 
to such
actions are therefore simple and have a unique tracial state.
\end{rema}

\section{Combinatorial properties of the sequence $\{a_{n}\}$} \label{sec:comb}

\noindent As mentioned before, one of our main goals is to describe $a_n$ as the rank of the 
$K$-groups of $\mathscr{A}_{n,\theta}$. Since $a_n=\sum_{r=0}^n a_{n,r}$, it makes sense to first study $a_{n,r}$.
So we begin by finding some combinatorial properties of $a_{n,r}$, which is the 
rank of $\ker(\wedge^r \hat\sigma-\mathrm{id})$ for $r=0,1,\ldots,n$, where $\hat\sigma$ is the automorphism of $\Bbb Z^n$ corresponding to the Anzai transformation $\sigma$ on $\Bbb T^n$, and is represented by the integer matrix $\mathsf{S}_n$ as in the beginning of Section \ref{sec:Anzai}. In fact, we will show 
that $a_{n,r}$ equals the number of partitions of $[r(n+1)/2]$ to $r$ distinct positive
integers not greater than $n$.  To do this, we will use properties of the
irreducible representations of the simple Lie algebra 
$\mathfrak{sl}(2,\Bbb{C})$.

\subsection{Connections with representation theory of $\mathfrak{sl}(2,\mathbb{C})$}
The automorphism $\hat\sigma$ is realized through its action on the basis $\{e_1,\ldots,e_n\}$ of $\Bbb Z^n$,
where $e_i:=[z_i]$ for $i=1,\ldots,n$ as in Section \ref{sec:general}, and we have $\hat\sigma(e_i)=e_i+e_{i-1}$ with $e_0:=0$. Therefore introducing a new endomorphism of $\Bbb Z^n$ by $\hat\varphi:=\hat\sigma-\mathrm{id}$, we will get $$\hat\varphi(e_i)=e_{i-1}.$$ This is
precisely a relation that may be recognized as part of the data of the canonical representation $\pi_n$ of  the Lie algebra $\mathfrak{sl}(2,\mathbb{C})$ on a complex vector space $V$ with basis $\{e_1,\ldots,e_n\}$. More precisely, we have $$\hat\varphi=\pi_n(f),$$
where the canonical representation $\pi_n$ and the (third)
basis element $f\in\mathfrak{sl}(2,\mathbb{C})$ are defined in 
Appendix \ref{sl_2}. The endomorphism $\hat\varphi$ induces a derivation on $\Lambda^{r}V$, which is defined by
$$\mathcal{D}^{\hspace{1pt}r}{\hat\varphi}(x_1 
\wedge\ldots\wedge x_r)=\sum_{i=1}^{r}x_1 \wedge\ldots
\wedge{\hat\varphi}(x_i)\wedge\ldots\wedge x_r$$
for $r=2,\ldots,n$ and  $x_i \in V$, and by setting $\mathcal{D}^{\hspace{1pt}0}{\hat\varphi}:=0$, 
and $\mathcal{D}^{\hspace{1pt}1}{\hat\varphi}:=\hat\varphi$ (see Appendix \ref{endomorphism}). Then the following result
states that $a_{n,r}$ equals the nullity of the linear mapping $\mathcal{D}^{\hspace{1pt}r}{\hat\varphi}$.

\begin{prop}
   Let $\sigma$ be an Anzai transformation on $\Bbb{T}^n$ and $\sigma_{*}$
   be the corresponding induced homomorphism on 
$K_*(\mathcal{C}(\mathbb{T}^n))=\Lambda^*\Bbb
   {Z}^n$. Let $\hat\sigma$ be the restriction of $\sigma_{*}$ to 
$\mathbb{Z}^n$ and consider
   the linear mapping $\hat\sigma\otimes{1}$ on $V:=\Bbb{Z}^n 
\otimes\Bbb{C}$. Take
   ${\hat\varphi}=\hat\sigma\otimes{1}-\mathrm{id}$ and 
$\mathcal{D}^{\hspace{1pt}r}{\hat\varphi}$ 
as above. Then
     $$a_{n,r}=\mathrm{rank}\ker(\wedge^{r}\hat\sigma-\mathrm{id})=
\dim\ker\mathcal{D}^{\hspace{1pt}r}{\hat\varphi}.$$
   \end{prop}
\begin{proof}
Since $\hat\varphi$ is a nilpotent mapping, we can use Corollary \ref{sim} to conclude that $\wedge^{r} {(\hat\sigma\otimes 1)}-\mathrm{id} 
\sim
\mathcal{D}^{\hspace{1pt}r}{\hat\varphi}$. Therefore
\begin{equation*}
\mathrm{rank}\ker(\wedge^{r}\hat\sigma-\mathrm{id})=\dim\ker(\wedge^{r}
(\hat\sigma\otimes 1)-\mathrm{id})=
\dim\ker\mathcal{D}^{\hspace{1pt}r}{\hat\varphi}.
\end{equation*}
\end{proof}

\begin{nota}\label{nota:3}
Let $n,k,r$ be positive integers. Then $P(n,r,k)$ denotes the number
of partitions of $k$ to $r$ distinct positive integers not greater than $n$. 
In other
words
$$P(n,r,k)=\textnormal{card}\{(i_1,\ldots,i_r)\mid i_1+\ldots+i_r=k, 
1 \leq i_1<\ldots<i_r
\leq n \}.$$
By convention, we set $P(n,0,0)=1$ and $P(n,r,0)=P(n,0,k)=0$ for $r,k \geq 1$.
\end{nota}

\noindent We are ready now to state the main result of this section.

\begin{theo}\label{a_{n,r}}
With the above notation, $a_{n,r}=P(n,r,[r(n+1)/2])$, where $[x]$ 
denotes the greatest integer not greater than $x$. In particular,
$$a_n=\sum_{r=0}^n P(n,r,[\frac{r(n+1)}{2}]).$$
\end{theo}
\begin{proof}
Let $\pi_n:\mathfrak{sl}(2,\mathbb{C})\to\mathfrak{gl}(V)$ be the canonical representation of 
the Lie algebra $\mathfrak{sl}(2,\mathbb{C})$ on the $n$-dimensional complex vector space $V$, and extend $\pi_n$
to $\pi^r_n:\mathfrak{sl}(2,\mathbb{C})\rightarrow \mathfrak{gl}(\Lambda^r V)$ 
with $\pi^1_n=\pi_n$. More precisely, for every $X \in \mathfrak{sl}(2,\mathbb{C})$
define
$$\pi^r_n(X)(v_1 \wedge \ldots\wedge v_r)=(\pi_n(X) v_1)\wedge v_2 
\wedge\ldots\wedge v_r
+\ldots +v_1 \wedge\ldots\wedge v_{r-1}\wedge (\pi_n(X)v_r).$$
This means that we have $\mathcal{D}^{\hspace{1pt}r}{\hat\varphi}=\pi^r_n(f).$
In particular,  $a_{n,r}$ is the nullity of $\pi^r_n(f)$ by the previous proposition.
Following Weyl's theorem (see Theorem \ref{Weyl}), since  the Lie algebra $\mathfrak{sl}(2,\mathbb{C})$ is semisimple the representation $\pi^r_n$ has to be completely reducible. This means we should have a decomposition $\Lambda^r V=
\oplus_{p=1}^N W_p$, where $W_p$'s are some $\pi^r_n$-invariant irreducible 
subspaces of $\Lambda^r V$. Moreover, the number $N$ of such subspaces is equal to $\dim E_0+\dim E_1$, where
$$E_j=\{v \in \Lambda^r V \mid \pi^r_n(h)\,v=j\,v\},\hspace{.3in} (j=0,1)$$ 
and $h$ is the first basis element of $\mathfrak{sl}(2,\mathbb{C})$ as in Appendix \ref{sl_2} (see Theorem \ref{rep}).
On the other hand, the number $N$ is equal to the nullity of $\pi^n_r(f)$. In fact, since $\pi^n_r|_{W_p}$
is an irreducible representation of $\mathfrak{sl}(2,\mathbb{C})$ on $W_p$, it is equivalent to the canonical 
representation of $\mathfrak{sl}(2,\mathbb{C})$ on $W_p$ by Theorem \ref{rep}. But the image of $f$ in the canonical representation has a $1$-dimensional kernel due to the part (c) of Proposition \ref{canonical}. So the nullity of 
$\pi^n_r(f)$ counts the number of $W_p$'s. Therefore $$a_{n,r}=\dim\ker\pi^r_n(f)=\dim E_0+\dim E_1.$$
To compute the last two terms, note that using Proposition \ref{canonical} we have $\pi_n(h)e_i=(2i-n-1)e_i$, which leads to
$$\pi^r_n(h)(e_{i_1}\wedge\ldots\wedge e_{i_r})=(2(i_1+\ldots+i_r)-r(n+1))
e_{i_1}\wedge\ldots\wedge e_{i_r}.$$
So for even $r(n+1)$ we have $E_1=\{0\}$ and $\dim E_0=P(n,r,r(n+1)/2)$, and for odd
$r(n+1)$ we have $E_0=\{0\}$ and $\dim E_1=P(n,r,r(n+1)/2-1)$. To summarize, we have established the following 
equalities
$$a_{n,r}=\dim \ker\mathcal{D}^{\hspace{1pt}r}{\hat\varphi}=\dim\ker\pi^r_n(f)=N=\dim E_0+\dim E_1 
=P(n,r,[r(n+1)/2]).$$
The desired formula for $a_n$ is immediate now by writing $a_n=\sum_{r=0}^n a_{n,r}$. 
\end{proof}
~\\
\noindent Using the previous theorem, we can prove that $\{a_n \}$ is a strictly increasing sequence.
We need a lemma first.
\begin{lemm}
$P(n+1,r,k+s)\geq P(n,r,k)$ for $s=0,1,\ldots,r$.
\end{lemm}
\begin{proof}
For $s=0$, the proof is clear. Now, let $1 \leq s \leq r$ and suppose that
$(j_1,\ldots,j_r)$ is a partition of $k$ such that $1 \leq j_1 <\ldots<j_r 
\leq n$. Define $i_q:=j_q$ for $1 \leq q \leq r-s$ and  $i_q:=j_q+1$
for $r-s+1 \leq q \leq r$. Then $(i_1,\ldots,i_r)$ is a partition of $k+s$ 
and
$1 \leq i_1<\ldots<i_r \leq n+1$. Thus $P(n+1,r,k+s)\geq P(n,r,k)$.
\end{proof}

\begin{prop}\label{prop:2}
$\{a_n \}$ is a strictly increasing sequence.
\end{prop}
\begin{proof}
First, note that $a_{n,0}=a_{n,n}=P(n,0,0)=P(n,n,n(n+1)/2)=1$, and from the 
previous theorem we have $a_n=\sum_{r=0}^n P(n,r,[r(n+1)/2])$. Fix $m\in\Bbb N$, and get
\begin{align*}
a_{2m+1}&=1+\sum_{r=0}^m P(2m+1,2r,2rm+2r)+\sum_{r=0}^{m-1} P(2m+1,2r+1,2rm+2r+m+1),
\end{align*}
\begin{align*}
a_{2m}&=\sum_{r=0}^m P(2m,2r,2rm+r)+\sum_{r=0}^{m-1} P(2m,2r+1,2rm+m+r)\\
&=1+\sum_{r=0}^{m-1} P(2m,2r,2rm+r)+\sum_{r=0}^{m-1} P(2m,2r+1,2rm+m+r),
\end{align*}
\begin{align*}
a_{2m-1}&=\sum_{r=0}^{m-1} P(2m-1,2r,2rm)+\sum_{r=0}^{m-1} 
P(2m-1,2r+1,2rm+m).
\end{align*}
Applying the previous lemma to the terms of the sums expressed above implies that $$a_{2m+1}>a_{2m}>a_{2m-1}.$$
\end{proof}

\subsection{Generating functions for the sequence $\{a_n\}$}
In this part, we express the rank of the $K$-groups of 
$\mathscr{A}_{n,\theta}$ as explicitly as possible. In fact, 
we present them as the constant terms in the polynomial 
expansions of certain functions. First of all, we need the 
following basic lemma.

\begin{lemm}
Let $P(n,r,k)$ denote the number of partitions of $k$ to $r$ distinct 
positive integers not greater than $n$. Then $P(n,r,k)$ is the coefficient of $u^r t^k$
in the polynomial expansion of $F_n(u,t):=\prod_{i=1}^n (1+u t^i)$. In other 
words, $$\sum_{r,k \geq 0}P(n,r,k)u^r t^k=\prod_{i=1}^n (1+u t^i).$$
\end{lemm}
\begin{proof}
\begin{align*}
\prod_{i=1}^n (1+u t^i)&=1+\sum_{r=1}^n
\underset{1 \leq i_1<\ldots<i_r \leq n}
{\sum_{\small{(i_1,\ldots,i_r)}}}
{(ut^{i_1})\ldots (ut^{i_r})}
=1+\sum_{r=1}^n \sum_{k \geq 1}P(n,r,k)u^r t^k
=\sum_{r,k \geq 0}P(n,r,k)u^r t^k.
\end{align*}
\end{proof}

\noindent Now, we have the following result for the rank $a_n$ of the $K$-groups
of $\mathscr{A}_{n,\theta}$.

\begin{theo}\label{theo:6}
Let $a_n=\mathrm{rank}\hspace{2pt}K_0(\mathscr{A}_{n,\theta})=
\mathrm{rank}\hspace{2pt}K_1(\mathscr{A}_{n,\theta})$. Then for a nonnegative 
integer $m$ we have\\

\begin{itemize}
\item[(i)] $a_{2m+1}$ is the constant term in the Laurent polynomial expansion of
$$\prod_{j=-m}^m (1+z^j),$$
\item[(ii)] $a_{2m}$ is the constant term in the Laurent polynomial expansion of
$$(1+z)\prod_{j=-m+1}^m (1+z^{2j-1}).$$
\end{itemize}
\end{theo}
\begin{proof}
We know that $a_n=\sum _{r=0}^n a_{n,r}$ and $a_{n,r}=P(n,r,[r(n+1)/2])$ by Theorem \ref{a_{n,r}}.
We have $a_{2m+1}=\sum_{r=0}^{2m+1}P(2m+1,r,r(m+1)).$
Now, take $y=u t^{m+1}$ and use the preceding lemma to get
$$F_{2m+1}(u,t)=F_{2 m+1}(y t^{-m-1},t)=\prod_{i=1}^{2m+1}(1+y t^{i-m-1})=
\sum_{r,k \geq 0}P(2m+1,r,k)y^r t^{k-r(m+1)}.$$
In particular, we get the following identity for $y=1$
$$\prod_{i=1}^{2m+1}(1+t^{i-m-1})=
\sum_{r,k \geq 0}P(2m+1,r,k)t^{k-r(m+1)},$$
or equivalently, by setting $z=t$ and $j=i-m-1$ we have
$$\prod_{j=-m}^{m}(1+z^{j})=
\sum_{r,k \geq 0}P(2m+1,r,k)z^{k-r(m+1)}.$$
In particular, the constant term in the Laurent polynomial expansion of $\prod_{j=-m}^{m}(1+z^{j})$
is obtained when we take the sum of those terms for which $k=r(m+1)$ holds, namely
$$\sum_{r=0}^{2m+1}P(2m+1,r,r(m+1)),$$
which is precisely the expression for $a_{2m+1}$.\\

\noindent For part (ii), write 
\begin{align*}
a_{2m}&=\sum_{r=0}^{2m}P(2m,r,[r(m+\frac{1}{2})])
          =\sum_{r=0}^m P(2m,2r,r(2m+1))+\sum_{r=0}^{m-1} 
P(2m,2r+1,2rm+m+r)\\
		&  =:A_m + B_m.
\end{align*}
Let us determine $A_m$ first. Note that using the preceding lemma we have
\begin{align*}
\frac{1}{2}\{\prod_{i=1}^{2m}(1+ut^i)+\prod_{i=1}^{2m}(1-ut^i)\}&=
\sum_{r,k \geq 0}P(2m,r,k)\{\frac{1+(-1)^r}{2}\}u^r t^k
=\sum_{r,k \geq 0}P(2m,2r,k)u^{2r} t^k.
\end{align*}
If we define $y:=u^2 t^{2m+1}$, we have the following identity
$$
\frac{1}{2}
\{\prod_{i=1}^{2m}(1+y^{\frac{1}{2}}t^{i-(m+\frac{1}{2})})
+\prod_{i=1}^{2m}(1-y^{\frac{1}{2}}t^{i-(m+\frac{1}{2})})\}=
\sum_{r,k \geq 0}P(2m,2r,k)y^r t^{k-r(2m+1)},
$$
which for $y=1$ yields
$$\frac{1}{2}\{\prod_{i=1}^{2m}(1+t^{i-(m+\frac{1}{2})})
+\prod_{i=1}^{2m}(1-t^{i-(m+\frac{1}{2})})\}=
\sum_{r,k \geq 0}P(2m,2r,k)t^{k-r(2m+1)}.$$
Hence $A_m$ is the constant term in the polynomial expansion of
$$\frac{1}{2}\{\prod_{i=1}^{2m}(1+t^{i-(m+\frac{1}{2})})
+\prod_{i=1}^{2m}(1-t^{i-(m+\frac{1}{2})})\}.$$
Similarly, for $B_m$ we have
\begin{align*}
\frac{1}{2}\{\prod_{i=1}^{2m}(1+ut^i)-\prod_{i=1}^{2m}(1-ut^i)\}&=
\sum_{r,k \geq 0}P(2m,r,k)\{\frac{1-(-1)^r}{2}\}u^r t^k
=\sum_{r,k \geq 0}P(2m,2r+1,k)u^{2r+1} t^k.
\end{align*}
If we define $y^2:=u^2 t^{2m+1}$, we have the following identities
\begin{multline*}
\frac{1}{2}\{\prod_{i=1}^{2m}(1+y^{\frac{1}{2}}t^{i-(m+\frac{1}{2})})
-\prod_{i=1}^{2m}(1-y^{\frac{1}{2}}t^{i-(m+\frac{1}{2})})\}\\
=\sum_{r,k \geq 0}P(2m,2r+1,k)y^{2r+1} t^{k-(2rm+r+m)-\frac{1}{2}}\\
=t^{-\frac{1}{2}}\sum_{r,k \geq 0}P(2m,2r+1,k)y^{2r+1} t^{k-(2rm+r+m)},
\end{multline*}
which for $y=1$ yields
$$
\frac{t^{\frac{1}{2}}}{2}\{\prod_{i=1}^{2m}(1+t^{i-(m+\frac{1}{2})})
-\prod_{i=1}^{2m}(1-t^{i-(m+\frac{1}{2})})\}=\\
\sum_{r,k \geq 0}P(2m,2r+1,k)t^{k-(2rm+r+m)}.
$$
Hence $B_m$ is the constant term in the polynomial expansion of
$$\frac{t^{\frac{1}{2}}}{2}\{\prod_{i=1}^{2m}(1+t^{i-(m+\frac{1}{2})})
-\prod_{i=1}^{2m}(1-t^{i-(m+\frac{1}{2})})\}.$$
Therefore $a_{2m}=A_m+B_m$ is the constant term in the polynomial
expansion of
$$
\frac{1}{2}\{\prod_{i=1}^{2m}(1+t^{i-(m+\frac{1}{2})})
+\prod_{i=1}^{2m}(1-t^{i-(m+\frac{1}{2})})\\+\frac{t^{\frac{1}{2}}}{2}
\prod_{i=1}^{2m}(1+t^{i-(m+\frac{1}{2})})-\frac{t^{\frac{1}{2}}}{2}
\prod_{i=1}^{2m}(1-t^{i-(m+\frac{1}{2})})\},
$$
or equivalently, the constant term in the polynomial
expansion of
$$
\frac{1}{2}\{(1+z)\prod_{i=1}^{2m}(1+z^{2i-(2m+1)})
+(1-z)\prod_{i=1}^{2m}(1-z^{2i-(2m+1)})\},
$$
which equals the constant term in the Laurent polynomial
expansion of
$$(1+z)\prod_{j=-m+1}^{m}(1+z^{2j-1}).$$

\end{proof}
\noindent Thanks to this theorem, one can compute $a_n$ for large values of $n$ using a computer algebra program.
Many more terms are also available online at OEIS (The Online Encyclopedia of Integer Sequences at 
\texttt{www.oeis.org}). Moreover, as the following corollaries suggest, such recognitions as constant terms of certain Laurent polynomials opens the door to finding even more interesting combinatorial properties of the sequence $\{a_n\}$, which have been of interest to Erd\H{o}s, J. H. van Lint and R. C. Entringer to name a few (\textit{cf}. \cite{pE65,jhVL67,rcE68}).

\begin{coro}
Let $n$ be a nonnegative integer. 
\begin{itemize}
\item[(i)]
The integer $a_{2n+1}$ is the number of solutions of the equation
$$\sum_{k=-n}^{k=n} k\, \epsilon_k=0,$$ where $\epsilon_k=0$ or $1$ for $-n\le k\le n$. In other words,
$a_{2n+1}$ is the number of ways that a sum of integers between $-n$ and $n$ (with no repetitions) equals to $0$. 
\item[(ii)]
The integer $a_{2n}$ is the number of solutions of the equation
$$\sum_{k=-n+1}^{k=n} (2k-1)\, \epsilon_k=0~ {\text or}~1,$$
where $\epsilon_k=0$ or $1$ for $-n+1\le k\le n$. In other words,
$a_{2n}$ is the number of ways that a sum of half-integers between $-n+1/2$ and $n-1/2$ (with no repetitions) equals to $0$ or $1/2$. 
\end{itemize}
\end{coro}
\begin{proof}
Using Theorem \ref{theo:6}, the number $a_{2n+1}$ is the constant term in the Laurent polynomial expansion of 
$\prod_{k=-n}^n (1+z^k)$, which is a finite sum of the form $\sum A(n,m)z^m$. Obviously, the integer coefficient $A(n,m)$
is the number of all possible combinations from the terms $z^{-n},\ldots,z^0,\ldots,z^n$, whose product makes a $z^m$. In other words, by putting $\epsilon_k=1$ when $z^k$ contributes to such a product making a $z^m$, and $\epsilon_k =0$ otherwise, we conclude that 
$$A(n,m)=\#\{(\epsilon_{-n},\ldots,\epsilon_0,\ldots,\epsilon_n)\in\{0,1\}^{2n+1}:\sum_{k=-n}^{k=n} k\, \epsilon_k=m\}.$$
In particular, the constant term of the Laurent polynomial expansion is $A(n,0)$, and we have
$a_{2n+1}=A(n,0)$. This proves part (i). Fort part (ii), we use the same idea for the Laurent polynomial expansion of
$$(1+z)\prod_{k=-m+1}^{m}(1+z^{2k-1})=\prod_{k=-m+1}^{m}(1+z^{2k-1})+z\prod_{k=-m+1}^{m}(1+z^{2k-1})$$
as suggested by part (ii) of Theorem \ref{theo:6}.
\end{proof}

\

\noindent 
J. H. van Lint in \cite{jhVL67} answered a question of Erd\H{o}s by determining the asymptotic behavior 
of $$A(n,0)=\#\{(\epsilon_{-n},\ldots,\epsilon_0,\ldots,\epsilon_n)\in\{0,1\}^{2n+1}:\sum_{k=-n}^{k=n} k\, \epsilon_k=0\}.$$
The idea in his proof is as follows. Since $A(n,0)$ is the constant term of the Laurent polynomial expansion
of  $\prod_{k=-n}^n (1+z^k)$, we can compute it as the Cauchy integral 
$$\frac{1}{2\pi i}\oint_C\frac{\prod_{k=-n}^n (1+z^k)}{z}\,dz,$$
where $C$ denotes the unit circle. By parameterizing $C$ by $z=e^{2 i x}$ for 
$x\in[0,\pi]$, applying the elementary identity $(1+e^{2 i k x})(1+e^{-2 i k x})=4 \cos^2 kx$, and
a simple calculation we arrive at
$$A(n,0)=\frac{2^{2n+2}}{\pi}\int_0^{\frac{\pi}{2}}\prod_{k=1}^n\cos^2 kx\,dx.$$
We can then proceed by estimating the integrand near and far from $0$ using
some elementary inequalities, which lead to the asymptotic formula
$A(n,0)\thicksim(3/\pi)^\frac{1}{2} 2^{2n+1}n^{-\frac{3}{2}}$ \cite{jhVL67}. This will immediately  give
the asymptotic behavior of the sequence $\{a_{2n+1}\}$ by the previous corollary. One can 
adapt the arguments used by J. H. van Lint to obtain a similar asymptotic behavior for the 
sequence $\{a_{2n}\}$ by estimating the corresponding integral
$$\frac{2^{2n+2}}{\pi}\int_0^{\frac{\pi}{2}}\cos^2x\prod_{k=1}^n\cos^2 (2k-1)x\,dx,$$
which leads to the asymptotic formula $a_{2n}\thicksim(3/\pi)^\frac{1}{2} 2^{2n}n^{-\frac{3}{2}}$.
This gives rise to the following result.
 
\begin{coro}
$\displaystyle{a_n \thicksim \sqrt{\frac{24}{\pi}}\,2^n n^{-\frac{3}{2}}}$ when $n 
\rightarrow \infty$. In particular, $\displaystyle\lim_{n\to\infty}\frac{a_{n+1}}{a_n}=2$.
\end{coro}

\section{The positive cone of $K_0({\mathscr F}_{\theta,{\boldsymbol f}})$}\label{order}

\noindent In this section, we generalize a result of Kodaka on the order structure of the group $K_0$
of the crossed product by a Furstenberg transformation on the $2$-torus \cite[Theorem 5.2]{kK00}.
However, our approach is different, and follows the general guidelines of \cite[Lemma 3.1]{ncp07}.
We remind the reader that for a $C^*$-algebra $A$ the positive cone of $K_0(A)$ is the set $K_0(A)_+=\{[q]\in K_0(A): q\in \mathcal {P}_\infty(A)\}$, where $\mathcal {P}_\infty(A)$ is the set of all projections in matrix algebras over $A$.
Also, any positive trace $\tau$ on a $C^*$-algebra $A$ induces a group homomorphism $\tau_*:K_0(A)\to\Bbb R$.
As was indicated in the Introduction, when the Furstenberg transformation 
$\varphi_{\theta,{\boldsymbol f}}$ is minimal and uniquely ergodic, using the results of H. Lin and N. C. Phillips in \cite{LP10} the transformation group $C^*$-algebra ${\mathscr F}_{\theta,{\boldsymbol f}}$ is classifiable by its
Elliott invariant, and the order of $K_0({\mathscr F}_{\theta,{\boldsymbol f}})$ is determined by 
the unique tracial state $\tau$ on ${\mathscr F}_{\theta,{\boldsymbol f}}$ \cite{qLP97,ncP07}. The fact that $\tau_*K_0({\mathscr F}_{\theta,{\boldsymbol f}})=\Bbb Z + \Bbb Z\theta$ was first proved in the unpublished thesis of R. Ji \cite{rJ86}. However,  we will study the effect of the trace on the order structure of $K_0$ using R. Exel's machinery of rotation numbers \cite{rE87}.\\

%\noindent The following result was proved by R. Ji \cite[Theorem 2.23]{rJ86}, but we will prove it using the 
%R. Exel's machinery of Rotation Numbers \cite{rE87}.
\noindent For a $C^*$-algebra $A$, we denote by $U_p(A)$ the set of unitary elements of $M_p(A)$. The following lemma is well known, but it is convenient to state and prove it for self-containment of the paper.

\begin{lemm}\label{Bott}
Let $A$ and $B$ be unital $C^*$-algebras and let $A\otimes B$ denote their minimal tensor product.
Suppose that $u\in U_p(A)$ and $v\in U_q(B)$, and let $\phi:\mathcal{C}(\Bbb T^2):\to M_{pq}(A\otimes B)$
be the unique homomorphism mapping the coordinate unitaries $z_1,z_2\in U(\mathcal{C}(\Bbb T^2))$ to the commuting
unitaries $u\,\otimes\, 1_q, 1_p\,\otimes\, v\in U_{pq}(A\otimes B)$, respectively. Let $b(u,v)\in K_0(A\otimes B)$ denote the Bott element of $u, v$ defined by $K_0(\phi)(b)$, where $\beta=[z_1]\wedge[z_2]$ is the Bott element in $K_0(\mathcal{C}(\Bbb T^2))$ so that $K_0(\mathcal{C}(\Bbb T^2))=\Bbb Z[1]+\Bbb Z\,\beta$. Then $\tau_*(b(u,v))=0$
for any tracial state $\tau$ on $A\otimes B$.
\begin{proof}
Since $\tau\circ\phi$ is a trace on $\Bbb T^2$, there exists a Borel probability measure $\mu$ on $\Bbb T^2$ such that
$$(\tau\circ\phi)(f)=\int_{\Bbb T^2} f(x)\,d\mu(x),~~~~f\in\mathcal{C}(\Bbb T^2).$$
Write $\beta=[p]-[q]$,  where $p, q$ are appropriate projections in some matrix algebra 
over $\mathcal{C}(\Bbb T^2)$, so we have  
$$\tau_*(b(u,v))=\tau_*(K_0(\phi)(\beta)=(\tau\circ\phi)_*(\beta)=
\int_{\Bbb T^2}\mathrm{Tr}(p(x))-\mathrm{Tr}(q(x))\,d\mu(x).$$
It is well known that for the Bott element $b$ we have $\mathrm{Tr}(p(x))-\mathrm{Tr}(q(x))=0$,
namely, the projections $p(x)$ and $q(x)$ have the same rank for all $x\in\Bbb T^2$, and this common rank 
does not depend on $x$ since $\Bbb T^2$ is connected (in fact, they are rank one projections).
This can be proved either by a calculation of the traces of the projections $p(x)$ and $q(x)$ explicitly ({\it cf}. \cite[p. 7]{AP89}), or  
by using the  naturality in the K{\"u}nneth formula for $\Bbb T^2$, which shows that the image under 
any point evaluation of $\beta$ is zero. Briefly speaking, the map $x\mapsto\mathrm{Tr}(p(x))-\mathrm{Tr}(q(x))$
belongs to $\mathcal{C}(\Bbb T^2,\Bbb Z)$, so it has to assume a constant integer, which we call $\dim\beta$.
In particular, $\dim\beta$ is invariant under the change of coordinate $(\zeta_1,\zeta_2)\mapsto(\zeta_1,\zeta_2^{-1})$,
whereas the naturality of the K{\"u}nneth homomorphism $\alpha_{1,1}:K_1(\mathcal{C}(\Bbb T))\otimes K_1(\mathcal{C}(\Bbb T))\to K_0(\mathcal{C}(\Bbb T^2))$, which maps $[z]\otimes[z]$ to $\beta$, implies that the Bott element $\beta$ will transform into $-\beta$ under this change of coordinates since $[z]\otimes[z^{-1}]=[z]\otimes(-[z])=-([z]\otimes[z])$.
This means $\dim \beta=0$.
\end{proof}

\end{lemm}

\noindent  We denote by $u$ the unitary in ${\mathscr F}_{\theta,{\boldsymbol f}}$ implementing the action generated by the transformation $\varphi_{\theta,{\boldsymbol f}}$ on $\Bbb T^n$ with irrational parameter $\theta$, and by $z_1$ the unitary in $\mathcal{C}(\Bbb T^n)$ defined by  $z_1(\zeta_1,\ldots,\zeta_n)=\zeta_1$ as in Section \ref{sec:general}. Then we have $uz_1u^{-1}=z_1\circ\varphi^{-1}_{\theta,{\boldsymbol f}}=e^{2\pi i\theta} z_1$ so that $C^*(u,z_1)\cong A_\theta$, the irrational rotation algebra. Let $p_\theta\in C^*(u,z_1)$ be a Rieffel projection of trace $\theta$ as in \cite{mR81}.
It is obvious that $\tau_*([1])=1$. On the other hand, since the restriction of $\tau$ on the $C^*$-subalgebra
$A_\theta\subseteq{\mathscr F}_{\theta,{\boldsymbol f}}$ has to be the unique tracial state on $A_\theta$,
we have $\tau_*([p_\theta])=\theta$. The main result of this section will show that all the essential information about
the order structure of $K_0({\mathscr F}_{\theta,{\boldsymbol f}})$ is encoded in the embedding of $A_\theta$
in ${\mathscr F}_{\theta,{\boldsymbol f}}$.

\begin{theo}\label{theo:order}
Let $\varphi_{\theta,{\boldsymbol f}}$ be a minimal uniquely ergodic Furstenberg transformation on  $\Bbb{T}^n$ with $\theta\in(0,1)$ (e.g. when $\theta\in(0,1)\setminus\Bbb Q$ and each $f_i$ satisfies a uniform Lipschitz condition in $\zeta_i$ for $i=1,\ldots,n-1$).  Let  $a_n$ and $\mathscr{T}^0_{\boldsymbol f}$ denote, respectively, the rank and the torsion subgroup of 
$K_0({\mathscr F}_{\theta,{\boldsymbol f}})$ so that $K_0({\mathscr F}_{\theta,{\boldsymbol f}})\cong{\mathbb Z}^{a_n}\oplus\mathscr{T}^0_{\boldsymbol f}$. Then the isomorphism of $K_0({\mathscr F}_{\theta,{\boldsymbol f}})$ with this group can be chosen in such a way that 
\begin{itemize}
\item[(i)] the unique tracial state $\tau$ on ${\mathscr F}_{\theta,{\boldsymbol f}}$ induces the map $$\tau_*(a[1]+b[p_\theta],{\boldsymbol c},{\boldsymbol t})=a+b\,\theta$$ on  $K_0({\mathscr F}_{\theta,{\boldsymbol f}})$ for all  $(a[1]+b[p_\theta],{\boldsymbol c},{\boldsymbol t})\in(\Bbb Z[1]+\Bbb Z[p_\theta])\oplus\Bbb Z^{a_n-2}\oplus\mathscr{T}^0_{\boldsymbol f}
\cong{\mathbb Z}^{a_n}\oplus\mathscr{T}^0_{\boldsymbol f}$,
\item[(ii)] the positive cone $K_0({\mathscr F}_{\theta,{\boldsymbol f}})_+$ can be identified with 
$$\{(a[1]+b[p_\theta],{\boldsymbol c},{\boldsymbol t})\in(\Bbb Z[1]+\Bbb Z[p_\theta])\oplus\Bbb Z^{a_n-2}\oplus\mathscr{T}^0_{\boldsymbol f}: a+b\,\theta>0\}\cup\{0\}.$$
\end{itemize}
\end{theo}
\begin{proof}
 The idea of the proof is to show that there exists a generating set for the finitely generated abelian group $K_0({\mathscr F}_{\theta,{\boldsymbol f}})$ including $[1]$ and $[p_\theta]$ such that the induced homomorphism $\tau_*$ vanishes at all generators, except for $[1]$ and $[p_\theta]$ for which we have $\tau_*([1])=1$ and $\tau_*([p_\theta])=\theta$. Using  Theorem \ref{theo:1}  and setting 
$\alpha=\varphi_{\theta,{\boldsymbol f}}$ and $\alpha_j=K_j(\alpha)$ for $j=1,2$ we have
$$K_0({\mathscr F}_{\theta,{\boldsymbol f}})\cong\mathrm{coker}(\alpha_0-\mathrm{id})\oplus\ker(\alpha_1-\mathrm{id})
= \bigoplus_{r \ge 0}[\mathrm{coker}(\wedge^{2r}\hat{\alpha}-\mathrm{id})\oplus
\ker(\wedge^{2r+1}\hat{\alpha}-\mathrm{id})],$$
where $\hat\alpha$ is the restriction of $\alpha_1$ to the subgroup $\Bbb Z[z_1]+\ldots+\Bbb Z[z_n]$ of $K_1(\mathcal{C}(\Bbb T^n))$ as in Section \ref{sec:general}, and $z_j(\zeta,\ldots,\zeta_n)=\zeta_j$ of $\mathcal{C}(\Bbb T^n)$ for $j=1,\ldots,n$. Note that by Definition \ref{furst-anzai}, $\boldsymbol{f}=(f_1,\ldots,f_{n-1})$ consists of continuous functions
$f_{j-1}:\Bbb T^{j-1}\to\Bbb T$ for $j=2,\ldots,n$. First,  we ``linearize" each $f_{j-1}$ by finding the unique ``linear" function $$(\zeta_1\ldots,\zeta_{j-1})\mapsto\zeta_1^{b_{1j}}\ldots\zeta_{j-1}^{b_{j-1,j}},\hspace{.4in}(b_{j-1,j}\ne 0)$$ in the homotopy class of $f_{j-1}$. This allows us to calculate $\hat\alpha([z_j])$ by writing
$$
\hat\alpha([z_1])=[z_1\circ\varphi_{\theta,{\boldsymbol f}}^{-1}]=[e^{2\pi i\theta}z_1]=[z_1],
$$
$$
\hat\alpha([z_j])=[z_j\circ\varphi_{\theta,{\boldsymbol f}}^{-1}]=[f_{j-1}(z_1,\ldots,z_{j-1})z_j]=
[z_1^{b_{1j}}\ldots z_{j-1}^{b_{j-1,j}}z_j]=b_{1j}[z_1]+\ldots +b_{j-1,j}[z_{j-1}]+[z_j], 
$$
for $j=2,\ldots,n$. In other words, the integer
matrix of $\hat\alpha$ with respect to the basis $\{[z_1]\ldots,[z_n]\}$ of $\Bbb Z^n$ is  precisely in the form ($\heartsuit$) as in the proof of Theorem \ref{rank}. Now, we can realize $\alpha_0=\wedge^{\mathrm{even}}\hat\alpha$ and $\alpha_1=\wedge^{\mathrm{odd}}\hat\alpha$ to calculate the $K$-groups of ${\mathscr F}_{\theta,{\boldsymbol f}}$
as in Section \ref{sec:general}. \\

\noindent It is important to note that, by referring to the exact sequences (\ref{P-V}), (\ref{K_0}) and (\ref{K_1}) in Section 1, the isomorphic image of $\mathrm{coker}(\alpha_0-\mathrm{id})$ in $K_0({\mathscr F}_{\theta,{\boldsymbol f}})$ is precisely the image $\mathrm{im}\jmath_0$ of $K_{0}(\mathcal{C}(\Bbb{T}^n))$, and an isomorphic image of $\ker(\alpha_1-\mathrm{id})$ in $K_0({\mathscr F}_{\theta,{\boldsymbol f}})$ is obtained by finding 
the image of a splitting (injective) homomorphism $s:\ker(\alpha_1-\mathrm{id})\to K_0({\mathscr F}_{\theta,{\boldsymbol f}})$
for the exact sequence (\ref{K_0}) so that $\partial\circ s=\mathrm{id}$ on $\ker(\alpha_1-\mathrm{id})$.
Any such splitting homomorphism is obtained as follows: fix a basis $\{\gamma_1,\ldots,\gamma_q\}$ for the 
free finitely generated group $\ker(\alpha_1-\mathrm{id})$ of rank $q$, and find elements $\nu^{(0)}_1,\ldots,\nu^{(0)}_q\in K_0({\mathscr F}_{\theta,{\boldsymbol f}})$ such that $\partial\nu^{(0)}_j=\gamma_j$ for $j=1,\ldots,q$. Then define $s(\sum_j m_j\gamma_j)=\sum_j m_j\nu^{(0)}_j$ for $m_j\in\Bbb Z$. Clearly, $K_0({\mathscr F}_{\theta,{\boldsymbol f}})=\mathrm{im}\,\jmath_0\oplus\mathrm{im}\,s$.\\

\noindent  Now, since $\wedge^0\hat\alpha=\mathrm{id}$ on $\Lambda^0\Bbb Z^n=\Bbb Z$, and 
$\wedge^1\hat\alpha=\hat\alpha$ on $\Lambda^1\Bbb Z^n=\Bbb Z^n$, we can write the
isomorphism
$$K_0({\mathscr F}_{\theta,{\boldsymbol f}})
\cong \Bbb Z\oplus\ker(\hat{\alpha}-\mathrm{id})\oplus\bigoplus_{r \ge 1}[\ker(\wedge^{2r+1}\hat{\alpha}-\mathrm{id})\oplus\mathrm{coker}(\wedge^{2r}\hat{\alpha}-\mathrm{id})].$$
In fact, a single generator for the isomorphic image of $\Bbb Z$ in $K_0({\mathscr F}_{\theta,{\boldsymbol f}})$ is $[1]$, and since $b_{j-1,j}\ne 0$ for $j=2,\ldots,n$ we have $\ker(\hat{\alpha}-\mathrm{id})=\Bbb Z e_1=\Bbb Z [z_1]$. It is easy to see that $\partial([p_\theta])=[z_1]$ (see the proposition in the appendix of \cite{mP80}).
Therefore there exists a basis $\{\gamma_1,\ldots,\gamma_q\}$ for $\ker({\alpha}_1-\mathrm{id})=\oplus_{r\ge 0}\ker(\wedge^{2r+1}\hat{\alpha}-\mathrm{id})$ with $\gamma_1=[z_1]$ and a splitting homomorphism 
$s:\ker(\alpha_1-\mathrm{id})\to K_0({\mathscr F}_{\theta,{\boldsymbol f}})$ with $s([z_1])=[p_\theta]$. Hence a single generator for the image of $\ker(\hat{\alpha}-\mathrm{id})$ in $K_0({\mathscr F}_{\theta,{\boldsymbol f}})$ is $[p_\theta]$. \\

\noindent It remains to study the effect of $\tau_*$ on the isomorphic image of $\oplus_{r\ge 1}\mathrm{coker}(\wedge^{2r}\hat{\alpha}-\mathrm{id})$, which contains the the torsion subgroup $\mathscr{T}^0_{\boldsymbol f}$, and the image of $\oplus_{r\ge 1}\ker(\wedge^{2r+1}\hat{\alpha}-\mathrm{id})$ in $K_0({\mathscr F}_{\theta,{\boldsymbol f}})$.
For more convenience, set $e_j:=[z_j]$ for $j=1,\ldots,n$ as in Section \ref{sec:general}.
First, we study the isomorphic image of $\oplus_{r\ge 1}\mathrm{coker}(\wedge^{2r}\hat{\alpha}-\mathrm{id})$.
We show that $\tau_*$ vanishes on this whole subgroup by showing, equivalently, that $\tau_*$
vanishes on the image of the subgroup $\oplus_{r\ge 1}\Lambda_{\Bbb Z}^{2r}(e_1,\ldots,e_n)\subset K_0(\mathcal{C}(\Bbb T^n))$
in $K_0({\mathscr F}_{\theta,{\boldsymbol f}})$ under the map $\jmath_0$. Let $\eta=e_{i_1}\wedge\ldots\wedge
e_{i_{2r}}\in K_0(\mathcal{C}(\Bbb T^n))$ for some $r\ge 1$ and $1\le i_1<\ldots<i_{2r}\le n$. We want to show that 
$\tau_*(\jmath_0(\eta))=0$, where $\jmath_0:=K_0(\jmath)$ and $\jmath:\mathrm{C}(\Bbb T^n)\to {\mathscr F}_{\theta,{\boldsymbol f}}$ is the natural embedding in the structure of the crossed product ${\mathscr F}_{\theta,{\boldsymbol f}}=\mathrm{C}(\Bbb T^n)\rtimes_\alpha\Bbb Z$. By K{\"u}nneth formula we have $\eta=b(u,z)$, where $u$ is a unitary in 
some matrix algebra over $\mathcal{C}(\Bbb T^{n-1})$ with $[u]=e_{i_1}\wedge\ldots\wedge e_{i_{2r-1}}\in K_1(\mathcal{C}(\Bbb T^{n-1}))$ and $z$ is the canonical unitary in $\mathcal{C}(\Bbb T)$ with $[z]=e_{i_{2r}}\in K_1(\mathcal{C}(\Bbb T))$. By Lemma \ref{Bott}, we have $$\tau_*(\jmath_0(\eta))=\tau_*(K_0(\jmath)(\eta))=
(\tau\circ\jmath)_*(\eta)=(\tau\circ\jmath)_*(b(u,z))=0.$$
Now, we study the isomorphic image of $\oplus_{r\ge 1}\ker(\wedge^{2r+1}\hat{\alpha}-\mathrm{id})$ in
$K_0({\mathscr F}_{\theta,{\boldsymbol f}})$. We will show that $\tau_*$ assumes only integer values on this
whole subgroup. In other words, if $\{\gamma_1,\ldots,\gamma_q\}$ is a basis for the $\ker(\alpha_0-\mathrm{id})$
as above such that $\gamma_1=[z_1]$ is a basis for $\ker(\hat\alpha-\mathrm{id})$ and $\{\gamma_2,\ldots,\gamma_q\}$ is a basis for $\oplus_{r\ge 1}\ker(\wedge^{2r+1}\hat{\alpha}-\mathrm{id})$, then $\tau_*(\nu^{(0)}_1)=\theta$,
and $\tau_*(\nu^{(0)}_j)=k_j$ for some $k_j\in\Bbb Z$ for $j=2,\ldots,q$, where $\nu^{(0)}_j$'s are chosen in $K_0({\mathscr F}_{\theta,{\boldsymbol f}})$ so that $\partial(\nu^{(0)}_j)=\gamma_j$ for $j=1,\ldots,q$ and
$\nu^{(0)}_1=[p_\theta]$ as above. To demonstrate this, we prove that the determinant of any unitary representing an element in the subgroup $\oplus_{r\ge 1}\Lambda^{2r+1}_{\Bbb Z}(e_1,\ldots,e_n)$ is the 
constant function $1$. Then the rotation number homomorphism $\rho^\mu_\alpha:\ker(\alpha_1-\mathrm{id})\to\Bbb T$ defined by R. Exel is the constant $1$ on the subgroup $\oplus_{r\ge 1}\ker(\wedge^{2r+1}\hat{\alpha}-\mathrm{id})$ \cite[Theorem VI.11]{rE87}, hence the trace will be integer-valued on this subgroup because $\exp(2\pi i\tau_*(\eta))=\rho^\mu_\alpha\circ\partial(\eta)$
for all $\eta\in K_0({\mathscr F}_{\theta,{\boldsymbol f}})$ \cite[Theorem V.12]{rE87}. To calculate the 
determinant on $\oplus_{r\ge 1}\Lambda^{2r+1}_{\Bbb Z}(e_1,\ldots,e_n)$, let $\gamma=e_{i_1}
\wedge\ldots\wedge e_{i_{2r+1}}\in K_1(\mathcal{C}(\Bbb T^{n}))$, set $\eta=e_{i_1}\wedge\ldots\wedge
e_{i_{2r}}\in K_0(\mathcal{C}(\Bbb T^{n-1}))$, and write $\eta=[p]-[q]$ as above. Then $e_{i_{2r+1}}=[z]$
for the canonical unitary $z$ of $\mathcal{C}(\Bbb T)$, and using the K{\"u}nneth formula we have
\begin{eqnarray*}\gamma=\eta\otimes[z]=([p]-[q])\otimes[z]=[p]\otimes[z]+[q]\otimes[z^{-1}]
=[((1-p)\otimes1+p\otimes z)((1-q)\otimes1+q\otimes z^{-1})]
\end{eqnarray*}
So, $\gamma=[\omega_1 \omega_2]$, where $\omega_1:=(1-p)\otimes1+p\otimes z$ and 
$\omega_2:=(1-q)\otimes1+q\otimes z^{-1}$ are unitaries in some matrix algebra of the same size
over $\mathcal{C}(\Bbb T^n)$. Since for all $x\in\Bbb T^{n-1}$ the projections $p(x)$ and $q(x)$ have 
the same rank $\rho$ as in the proof of the previous lemma, we have the following equivalence 
of projections in some matrix algebra $M_l(\Bbb C)$
$$p(x)\sim\textbf{1}_\rho\oplus\textbf{0}_{\l-\rho}\sim q(x),$$
where $\textbf{1}_m, \textbf{0}_m$ denote the identity and the zero matrix of order $m$, respectively, and 
$\oplus$ is the direct sum of matrices. This implies the following 
unitary equivalence of projections in $M_{2l}(\Bbb C)$
$$p(x)\oplus \textbf{0}_l\sim_u \textbf{1}_\rho\oplus \textbf{0}_{2l-\rho}\sim_u q(x)\oplus \textbf{0}_l.$$

\noindent In particular, we conclude the following unitary equivalence of unitary matrices for all $(x,\zeta)\in\Bbb T^{n-1}\times\Bbb T$
$$\omega_1(x,\zeta)\oplus \textbf{1}_l\sim_u\zeta \textbf{1}_\rho\oplus \textbf{1}_{2l-\rho},\hspace{.4in}\omega_2(x,\zeta)\oplus \textbf{1}_l\sim_u\zeta^{-1} \textbf{1}_\rho\oplus \textbf{1}_{2l-\rho}.$$

\noindent Therefore $\mathrm{Det}\,\omega_1(x,\zeta)=\zeta^\rho$ and $\mathrm{Det}\,\omega_2(x,\zeta)=\zeta^{-\rho}$, hence $\mathrm{Det}\,(\omega_1\omega_2)(x,\zeta)=1$, for all $(x,\zeta)\in\Bbb T^{n-1}\times\Bbb T$.
This implies that $\mathrm{Det}_*(\gamma)=1\in[\Bbb T^n,\Bbb T]$, where $[\Bbb T^n,\Bbb T]$ denotes the set of homotopy classes of continuous functions from $\Bbb T^n$ to $\Bbb T$  (see Definition VI.8 and Proposition VI.9 of \cite{rE87}).\\

\noindent Finally, by setting $\nu_1:=\nu^{(0)}_1=[p_\theta]$ and $\nu_k:=\nu^{(0)}_j-k_j[1]$ for $j=2,\ldots,q$ so that $\tau_*(\nu_1)=\theta$ and $\tau_*(\nu_j)=0$ for $j=2,\ldots,q$, we can form a generating set with the desired property for $K_0({\mathscr F}_{\theta,{\boldsymbol f}})$ by taking the union of $\{\nu_1,\ldots \nu_q\}$ and a generating
set including $[1]$ for the isomorphic image of $\mathrm{coker}(\alpha_0-\mathrm{id})$. This proves part (i).\\

\noindent For part (ii), we use part (i) together with the fact that the order on $K_0({\mathscr F}_{\theta,{\boldsymbol f}})$ is determined by the effect of the unique tracial state $\tau$ because $\Bbb T^n$ is a finite dimensional infinite compact metric
space and $\varphi_{\theta,{\boldsymbol f}}$ is a minimal homeomorphism of $\Bbb T^n$ (see Theorem 5.1(1) of \cite{qLP97} or Theorem 4.5(1) of \cite{ncP07}).

\end{proof}

\begin{coro}
Let $\varphi_{\theta,{\boldsymbol f}}$ be a minimal uniquely ergodic Furstenberg transformation on $\Bbb T^n$ as 
above. Then linearizing the functions $f_i:\Bbb T^i\to\Bbb T$ in $\boldsymbol{f}=(f_1,\ldots,f_{i-1})$ does not change 
the isomorphism class of the transformation group $C^*$-algebra ${\mathscr F}_{\theta,{\boldsymbol f}}$.
\end{coro}
\begin{proof}
Since $\varphi_{\theta,{\boldsymbol f}}$ is minimal, $\theta$ must be irrational. So the range of the unique
tracial state (by unique ergodicity) on $K_0({\mathscr F}_{\theta,{\boldsymbol f}})$ is dense in $\Bbb R$ as it is 
$\Bbb Z + \Bbb Z \theta$ by the above argument. Benefiting from the results of \cite{LP10}, such $C^*$-algebras are completely classifiable by their Elliott invariants, which remain unchanged (up to isomorphism) after the linearization process: linearizing does not change the isomorphism classes of the $K$-groups, and the previous theorem guarantees that the order structure of the group  $K_0$ is precisely the regular order inherited from $\Bbb R$ on $\Bbb Z + \Bbb Z \theta$ before and after linearization.
\end{proof}

  \section{Concluding remarks}

\noindent ${\bf I)}$ The method used in Section $1$ for computing $K$-groups of 
the transformation group $C^*$-algebras of homeomorphisms of the tori may be extended to more 
general settings. Let $G$ be a compact connected Lie group with torsion-free fundamental group $\pi_1(G)$.
(It is well known that the fundamental group of such spaces are finitely generated and abelian, so being torsion-free means $\pi_1(G)\cong\Bbb Z^l$, for some $l$.) Some important examples are any finite Cartesian products of the groups
$S^3$, $SO(2)$, $Sp(n)$, $U(n)$ and $SU(n)$. Then $K^*(G)$ is torsion-free and can be given the structure of a
  $\mathbb{Z}_2$-graded Hopf algebra over the integers \cite{lH67}. 
Moreover, regarded as a Hopf algebra, $K^*(G)$ is the exterior algebra on the module of the primitive elements, 
which are of degree 1. The module of the primitive elements of $K^*(G)$ may also be 
described as follows. Let $U(n)$ denote the group of unitary matrices of order $n$ 
and let $U:=\cup_{n=1}^\infty U(n)$ be the stable unitary group.  Any unitary 
representation $\rho:G \rightarrow U(n)$, by composition with the inclusion $U(n)\subset U$, defines a homotopy 
class $\beta(\rho)$ in $[G,U]=K^1(G)$. The module of the primitive elements in $K^1(G)$
  is exactly the module generated by all classes $\beta(\rho)$ of this type.
  If in addition, $G$ is semisimple and simply connected of rank $l$, there 
are $l$ basic irreducible representations $\rho_1,\ldots,\rho_l$, whose 
maximum weights
  $\lambda_1,\ldots,\lambda_l$ form a basis for the character group 
$\mathbf{\hat{T}}$
  of the maximal torus $\textbf{T}$ of $G$ and the classes 
$\beta(\rho_1),\ldots,
  \beta(\rho_l)$ form a basis for the module of the primitive elements in 
$K^1(G)$ and
  $K^*(G)=\Lambda^*(\beta(\rho_1),\ldots,\beta(\rho_l))$. In any case, to 
compute
  $K_*(\mathcal{C}(G)\rtimes_{\alpha}\mathbb{Z})$ it is sufficient to 
determine the homotopy classes of $\alpha\circ\rho$ for the irreducible representations 
$\rho$ of $G$ in terms of $\beta(\rho)$'s.\\

 \noindent ${\bf II)}$ There is a relation between the $K$-theory of transformation group
  \text{$C^*$-algebras} of the homeomorphisms of the tori and the topological $K$-theory of compact 
nilmanifolds. In fact, let $\alpha=(\boldsymbol{t},\mathsf{A})$ be an affine transformation on 
$\Bbb{T}^n$ satisfying the conditions (1) through (4) in Remark \ref{rema:3}. Then it 
has been shown in \cite{fjH63} that $\alpha$ is conjugate (in the group of affine transformations on
$\Bbb{T}^n$) to the transformation $\alpha'=(\boldsymbol{t'},\mathsf{A}')$, where $\mathsf{A}'$ has an upper triangular matrix, whose bottom right $k \times k$ corner is the identity matrix $\mathsf{I}_k$ and $\boldsymbol{t}'=(0,\ldots,0,t'_1,
\ldots,t'_k)$. The transformation $\alpha'$ is called a \emph{standard form} for $\alpha$ \cite{jP86}.
   Assume that $\alpha$ is given in standard form. Then J. Packer associates an \emph{induced flow} $(\Bbb{R}, N/\Gamma)$ to the flow $(\Bbb{Z},\Bbb{T}^n)$ generated by $\alpha$, where $N$ is a simply connected nilpotent Lie 
group of dimension $n+1$, the discrete group $\Gamma$ is a cocompact subgroup of $N$, and the action of $\Bbb{R}$ is 
given by translation on the left by $\exp sX$ for $s \in\mathbb{R}$ and some $X 
\in\mathfrak{n}$, the Lie algebra of $N$. One of the most important facts is
  that the \text{$C^*$-algebra} $\mathcal{C}(N/\Gamma)\rtimes_\beta\Bbb{R}$ 
corresponding to the
  induced flow is strongly Morita equivalent to
  $\mathcal{C}(\Bbb{T}^n)\rtimes_\alpha\Bbb{Z}$ 
  \cite[Proposition 3.1]{jP86}. Consequently, one has
  \begin{equation}\label{Packer}
  K_i(\mathcal{C}(\Bbb{T}^n)\rtimes_\alpha\Bbb{Z})\cong
  K_i(\mathcal{C}(N/\Gamma)\rtimes_\beta\Bbb{R})\cong 
K^{1-i}(N/\Gamma);~~~i=0,1.
  \end{equation}
  The second isomorphism here, is the Connes' Thom
  isomorphism. So the $K$-theory of 
$\mathcal{C}(\Bbb{T}^n)\rtimes_\alpha\Bbb{Z}$
  is converted to the topological $K$-theory of the compact nilmanifold 
$N/\Gamma$.
  Following the proof of Proposition 3.1 in \cite{jP86}, one can conclude 
that for the special case of Anzai transformations we can take $N=\mathfrak{F}_{n-1}$ (the generic filiform Lie group 
of dimension $n+1$) and $\Gamma=\mathfrak{D}_{n-1}$, which were defined in \cite{kR06}.
  On the other hand, following \cite[Theorem 3.6]{jR84}, one has the 
 isomorphism
  \begin{equation}\label{nil}
  K_i(C^*(\Gamma))\cong K^{i+n+1}(N/\Gamma);~~~i=0,1.
  \end{equation}
  Combining (\ref{Packer}) and (\ref{nil}) one gets
  \begin{equation}
  K_i(\mathcal{C}(\Bbb{T}^n)\rtimes_\alpha\Bbb{Z})\cong 
K^{i+1}(N/\Gamma)\cong K_{i+n}(C^*(\Gamma));~~~i=0,1.
  \end{equation}
Using the above isomorphisms, one can relate the algebraic invariants of
the involved $C^*$-algebras and topological information of the 
corresponding nilmanifold. For example, since $N/\Gamma$ is a classifying space for 
$\Gamma$,
one has the following isomorphisms
\begin{equation}\label{cohomo}
H^*_{\mathrm{dR}}(N/\Gamma)\cong \Check{H}^*(N/\Gamma,\Bbb{R})\cong 
H^*(\Gamma,\Bbb{R})
\cong H^*(N,\Bbb{R})\cong H^*(\mathfrak{n},\Bbb{R}),
\end{equation}
where $H^*_{\mathrm{dR}}(N/\Gamma)$ denotes the de Rham cohomology of the 
manifold
$N/\Gamma$, $H^*(N/\Gamma,\Bbb{R})$ denotes the \v{C}ech cohomology of 
$N/\Gamma$ with
coefficients in $\Bbb{R}$, $H^*(\Gamma,\Bbb{R})$ denotes the group 
cohomology of $\Gamma$ with
coefficients in the trivial $\Gamma$-module $\Bbb{R}$, $H^*(N,\Bbb{R})$ 
denotes the
Moore cohomology group of $N$ (as a locally compact group) with coefficients 
in the trivial
Polish $N$-module $\Bbb{R}$,
and $H^*(\mathfrak{n},\Bbb{R})$ denotes the cohomology of the Lie algebra 
$\mathfrak{n}$
with coefficients in the trivial $\mathfrak{n}$-module $\Bbb{R}$. Now, using 
the Chern
isomorphisms 
$\mathrm{ch}_0:K^0(N/\Gamma)\otimes\Bbb{Q}
\rightarrow\Check{H}^{\mathrm{even}}(N/\Gamma,\Bbb{Q})$ and 
$\mathrm{ch}_1:K^1(N/\Gamma)\otimes\Bbb{Q}\rightarrow
\Check{H}^{\mathrm{odd}}(N/\Gamma,\Bbb{Q})$, one concludes that
the even and odd cohomology groups stated in (\ref{cohomo}) are all isomorphic to 
$\mathbb{R}^{k}$,
where $k$ is the (common) rank of the $K$-groups of
$\mathcal{C}(\Bbb{T}^n)\rtimes_\alpha\Bbb{Z}$ as in Corollary \ref{coro:3}.
As an example, if $N=\mathfrak{F}_{n-1}$, $\Gamma=\mathfrak{D}_{n-1}$, and
$\mathfrak{n}=\mathfrak{f}_{n-1}$, then the even and odd cohomology groups
stated in (\ref{cohomo}) are all isomorphic to $\mathbb{R}^{a_{n}}$,
where $a_{n}$ is the rank of the $K$-groups of $\mathscr{A}_{n,\theta}$ that
was studied in detail in Sections \ref{sec:rank} and \ref{sec:comb}. 
Conversely, one may use the topological
tools for $N/\Gamma$ to get some information about
$\mathcal{C}(\Bbb{T}^n)\rtimes_\alpha\Bbb{Z}$ and $C^*(\Gamma)$.
For example, we know that $N/\Gamma$ as a compact nilmanifold
can be constructed as a principal $\Bbb{T}$-bundle over a lower
dimensional compact nilmanifold \cite{vG97}. Then we can compute the 
topological
$K$-groups of $N/\Gamma$ using the six term Gysin exact sequence
\cite[IV.1.13, p. 187]{mK78}. As an example, one can see that 
$\mathfrak{F}_{n-1}/\mathfrak{D}_{n-1}$
is a principal \text{$\Bbb{T}$-bundle} over $\mathfrak{F}_{n-2}/\mathfrak{D}_{n-2}$, 
and the corresponding
Gysin exact sequence is in fact the topological version of the 
Pimsner-Voiculescu exact sequence for the crossed product
$\mathscr{A}_{n,\theta}\cong 
\mathscr{A}_{n-1,\theta}\rtimes_\alpha\mathbb{Z}$ as in
Theorem 2.1(d) in \cite{kR06}.\\

\noindent\textbf{Acknowledgment.}
I would like to thank Graham Denham, George
 Elliott, Herve Oyono-Oyono, and Tim
 Steger for very helpful discussions. I would like to thank 
 N. C. Phillips for his thoughtful suggestions and for bringing 
 the possibility of generalizing Lemma 3.1 of \cite{ncp07} to my attention.
 I am grateful to Alan T. Paterson for reading the paper and 
 suggesting several helpful comments. I would also like 
 to thank Paul Milnes and the Department of Mathematics 
 of the University of Western Ontario, where some parts of
 this work were done.
 
\appendix
\addcontentsline{toc}{chapter}{Appendices}
\section{The Smith normal form}\label{smith}
\noindent The Smith normal form is a very important tool for studying integer matrices. We refer to \cite{mN72} for this interesting topic and its applications.

\begin{defi}\label{equiv}
Let $\hat\alpha,\hat\beta\in\mathrm{End}(\Bbb{Z}^m)$. We say that 
$\hat\alpha$ is equivalent to
$\hat\beta$ over $\mathbb{Z}$ (and write $\hat\alpha$ \textnormal{equiv} 
$\hat\beta$)
if there exist $\hat{u},\hat{v} \in \mathrm{Aut}(\Bbb{Z}^m)$ such that
$\hat{u} \circ \hat\alpha\circ \hat{v}=\hat\beta$. Similarly, if 
$\mathsf{A}$ and $\mathsf{B}$ are integer $m \times m$ matrices,
$\mathsf{A}$ is equivalent to $\mathsf{B}$ if there exist ${\sf U,V} 
\in\mathrm{GL}(m,\Bbb{Z})$ such that
$\sf UAV=B$.
\end{defi}
\noindent Recall that $\hat\alpha$ \textnormal{equiv} $\hat\beta$, if and only if
$\mathrm{coker}\hat\alpha\cong
\mathrm{coker}\hat\beta$, if and only if $\hat\alpha$ and $\hat\beta$ have the same 
Smith normal form.
Also, $\mathsf{A}$ equiv $\mathsf{B}$ if and only if $\mathsf{B}$ is obtainable from 
$\mathsf{A}$ by a finite number of elementary
operations. An elementary operation on an integer matrix is one of the 
following types:
interchanging two rows (or two columns), adding an integer multiple of one 
row (or column) to
another,  and multiplying a row (or column) by $-1$.\\

\noindent Now, we recall a fundamental theorem for integer matrices ({\it cf}. \cite[p. 26]{mN72}).

\begin{theo}[Smith Normal Form]\label{theo:smith}
Let $\mathsf{A}$  be an $m\times m$ integer matrix. Then 
$\mathsf{A}$ is equivalent to a diagonal matrix 
${\rm diag}(d_1,\ldots,d_r,0,\ldots,0)$, where $r$ is the rank of $\mathsf{A}$, and the integers $d_1,\ldots,d_r$ satisfy 
$d_i|d_{i+1}$ for $i=1,\ldots,r-1$.
\end{theo}

\section{Nilpotent and unipotent linear mappings}\label{nilpotent}
\begin{defi}
Let $V$ be a (complex) vector space. A mapping ${\hat\epsilon}\in 
\mathrm{End}_{\mathbb{C}}V$
is called nilpotent (respectively, unipotent) if ${\hat\epsilon}^k=0$ 
(respectively,
$({\hat\epsilon}-\mathrm{id})^k=0$) for some positive integer $k$. The minimum value 
of $k$ with this property is
called the \textit{degree} of ${\hat\epsilon}$, denoted $\deg({\hat\epsilon})$.
\end{defi}

\noindent As an example, every upper (respectively, lower) triangular matrix with zeros on the 
diagonal is nilpotent.
Also, the matrix $\mathsf{S}_n$ defined in Section \ref{sec:general} is a unipotent matrix of degree $n$. 
Note that all eigenvalues of
a nilpotent (respectively, unipotent) matrix are zeros (respectively, ones). In particular, 
every unipotent matrix is invertible, and every unipotent endomorphism is an automorphism.

\begin{coro}
Let $V$ be a finite dimensional complex vector space and ${\hat\epsilon}$ be 
a nilpotent (respectively, unipotent) endomorphism of $V$.
Then $\deg({\hat\epsilon})$ is equal to the maximum order of its Jordan 
blocks.
\end{coro}
\begin{proof}
It suffices to prove the statement for the nilpotent case. Since all the 
eigenvalues of ${\hat\epsilon}$ are zero, each Jordan
block is a zero matrix of order one or is of the form
$$\left(\begin{array}{ccccc}
0 & 1 & 0 & \cdots& 0\\
0 & 0 & 1 & &\vdots\\
0 &\ddots&\ddots&\ddots&0\\
\vdots&   & 0 & 0 & 1\\
0 &\cdots  & 0 & 0 & 0
\end{array}
\right)$$
which is a nilpotent matrix and its degree is the same as its order,
which is greater than 1. The rest of proof is clear.
\end{proof}

\begin{defi}
Let $V$ be a finite dimensional complex vector space and
${\hat\epsilon}\in \mathrm{End}_{\mathbb{C}}V$ be nilpotent (respectively, 
unipotent).
We say that ${\hat\epsilon}$ is \textit{of maximal degree} if 
$\deg({\hat\epsilon})=\dim V$.
\end{defi}

\begin{coro}\label{coro:similar}
Let $V$ be an \text{$n$-dimensional} complex vector space and
${\hat\epsilon}\in \mathrm{End}_{\mathbb{C}}V$ be nilpotent (respectively, 
unipotent).
Then $\deg({\hat\epsilon})\leq n$. If $\deg({\hat\epsilon})=n$,
then the Jordan normal form of ${\hat\epsilon}$ is the full Jordan block of 
order $n$ with $0$'s on the diagonal. In particular, all nilpotent (respectively, unipotent) matrices of maximal degree 
acting on $V$ are similar.
\end{coro}

\begin{proof}
Use the preceding corollary and the Jordan normal form theorem.
\end{proof}

\begin{example}\label{furst-matrix}
Let $\mathsf{B}=[b_{ij}]_{n \times n}$ be any upper triangular matrix, whose 
diagonal elements
are all zeros (respectively, ones), and whose entries $b_{i,i+1}$ for $i=1,\ldots,n-1$ 
are all nonzero.
Then $\mathsf{B}$ is nilpotent (respectively, unipotent) of maximal degree. In fact, 
let $n$ be the order of $\mathsf{B}$
and let $b:=\prod_{i=1}^{n-1} b_{i,i+1}$, which is a nonzero number. Then 
one can
easily see that $\mathsf{B}^n=0$ (respectively, $(\mathsf{B}-\mathsf{I})^n=0$) and 
$\mathsf{B}^{n-1}$ (respectively, $(\mathsf{B}-\mathsf{I})^{n-1}$)
is a matrix with $b$ appearing on the upper-right corner and zeros elsewhere.
So $\deg(\mathsf{B})=n$, i.e. $\mathsf{B}$ is of maximal degree.
\end{example}

\begin{lemm}\label{lemm:3}
Let $V$ be a complex vector space and ${\hat\epsilon}\in 
\mathrm{End}_{\mathbb{C}} V$
be nilpotent of degree $k$. Then
$\exp({\hat\epsilon})$ is unipotent of degree $k$. Moreover, 
$\exp({\hat\epsilon})-\mathrm{id}$ is similar to ${\hat\epsilon}$.
\end{lemm}
\begin{proof}
For the first part, we know that
$\exp(\hat\epsilon)-\mathrm{id}={\hat\epsilon}+{\hat\epsilon}^2/2!+\ldots+
{\hat\epsilon}^{k-1}/(k-1)!={\hat\epsilon}\omega$,
where $\omega:=\mathrm{id}+{\hat\epsilon}/2!+\ldots
+{\hat\epsilon}^{k-2}/(k-1)!$ commutes with ${\hat\epsilon}$ and is 
invertible since it is unipotent. So, $(\exp(\hat\epsilon)-\mathrm{id})^r
=({\hat\epsilon}\omega)^r={\hat\epsilon}^r \omega^r$ for all positive 
integers $r$. Thus $\exp({\hat\epsilon})-\mathrm{id}$ is unipotent with
the same degree of ${\hat\epsilon}$. For the second part, using the Jordan 
normal form of ${\hat\epsilon}$, it is sufficient
to prove the statement for the special case when ${\hat\epsilon}$ is a 
Jordan block with zeros on the diagonal. Since in this
case ${\hat\epsilon}$ is of maximal degree, by the first part, 
$\exp({\hat\epsilon})-\mathrm{id}$ is also of maximal degree. Therefore they are
similar by Corollary \ref{coro:similar}.
\end{proof}

\section{Endomorphisms and derivations of exterior algebras}\label{endomorphism}
\noindent We refer to the Chapter 5 of \cite{wG78} for general properties of exterior algebras and
mappings between them.\\

\noindent Let $V$ be a (complex) vector space and ${\hat\phi}:V \rightarrow 
V$ be a linear mapping. Then ${\hat\phi}$ can be extended in a unique way to a 
homomorphism
$\wedge^{*}{\hat\phi}:\Lambda^{*}V \rightarrow \Lambda^{*}V$ such that 
$\wedge^{*}{\hat\phi}(1)=1$, yielding
$$\wedge^{*}{\hat\phi}(x_1 \wedge\ldots\wedge x_p)={\hat\phi}(x_1) 
\wedge\ldots\wedge {\hat\phi}(x_p),\hspace{.5in}(x_i \in V).$$
Also, ${\hat\phi}$ can be extended in a unique way to a derivation 
$\mathcal{D}^{\hspace{1pt}*}{\hat\phi}:\Lambda^{*}V \rightarrow
\Lambda^{*}V$, yielding 
$$\mathcal{D}^{\hspace{1pt}*}{\hat\phi}(x_1 
\wedge\ldots\wedge x_p)=\sum_{i=1}^{p}x_1 \wedge\ldots
\wedge{\hat\phi}(x_i)\wedge\ldots\wedge x_p\hspace{.5in}(p \geq 2, x_i \in V).$$
We define $\wedge^{r}{\hat\phi}:=\wedge^{*}{\hat\phi}|_{\Lambda^{r}V}$ 
and $\mathcal{D}^{\hspace{1pt}r} 
{\hat\phi}:=\mathcal{D}^{\hspace{1pt}*}{\hat\phi}|_{\Lambda^{r}V}$
as induced linear mappings on the $r$-th exterior power of $V$ for $r \geq 0$. 
Then we have
$$\wedge^{*}{\hat\phi}=\bigoplus_{r \geq 0}\wedge^{r}{\hat\phi}\,,
\hspace{.4in}\mathcal{D}^{\hspace{1pt}*}{\hat\phi}
=\bigoplus_{r \geq 0}\mathcal{D}^{\hspace{1pt}r}{\hat\phi}\,.$$

\noindent One can easily show that $\wedge^*(\hat\phi\circ\hat\psi)=(\wedge^*\hat\phi)\circ(\wedge^*\hat\psi)$
and $\mathcal{D}^{\hspace{1pt}*}([\hat\phi,\hat\psi])=[\mathcal{D}^{\hspace{1pt}*}{\hat\phi},\mathcal{D}^{\hspace{1pt}*}{\hat\phi}]$ ({\it cf}. equations (5.20) and (5.25) in \cite{wG78}).
\begin{lemm}
With the above notation, if ${\hat\phi}:V \rightarrow V$ is nilpotent, 
then $\wedge^{r}{\hat\phi}$ and
$\mathcal{D}^{\hspace{1pt}r}{\hat\phi}$ are also nilpotent for $r \geq 1$. 
If $V$ is finite dimensional, then $\mathcal{D}^{\hspace{1pt}*}{\hat\phi}$ is nilpotent.
\end{lemm}
\begin{proof}
Assume that ${\hat\phi}^t=0$ for some $t \in \mathbb{N}$. We have
$(\wedge^r {\hat\phi})^t(x_1 \wedge\ldots\wedge 
x_r)={\hat\phi}^t(x_1)\wedge\ldots\wedge {\hat\phi}^t(x_r)=0$.
So, $(\wedge^r {\hat\phi})^t=0$, which means that $\wedge^r {\hat\phi}$ is nilpotent. For 
$\mathcal{D}^{\hspace{1pt}r}{\hat\phi}$, we know that
$\mathcal{D}^{\hspace{1pt}r}{\hat\phi}(x_1 \wedge\ldots\wedge 
x_r)=\sum_{i=1}^{r}x_1 \wedge\ldots
\wedge{\hat\phi}(x_i)\wedge\ldots\wedge x_r$ and one can easily deduce that
$$\mathcal{D}^{\hspace{1pt}r}{\hat\phi}^p(x_1 \wedge\ldots\wedge 
x_r)=\underset{(i_j \geq 0)}{\sum_{i_1+\ldots+i_r=p}}
\frac{p!}{(i_1)!\ldots(i_r)!}~{\hat\phi}^{i_1}x_1 
\wedge\ldots\wedge{\hat\phi}^{i_r}x_r\,.$$
Now since $i_1+\ldots+i_r=p$, there exists an $i_j$ with $i_j \geq p/r$. So, 
if $p \geq r t$ then ${\hat\phi}^{i_j}=0$ and $\mathcal{D}^{\hspace{1pt}r}{\hat\phi}^p=0$. Thus
$\mathcal{D}^{\hspace{1pt}r}{\hat\phi}$ is nilpotent.\\

\noindent For the next part, let $m:=\dim V$. Since 
$\mathcal{D}^{\hspace{1pt}*}{\hat\phi}=\bigoplus_{r \geq 0}^{m}
\mathcal{D}^{\hspace{1pt}r}{\hat\phi}$ and ${\hat\phi}_0=0$, from the first 
part we have $(\mathcal{D}^{\hspace{1pt}*}{\hat\phi})^{mt}=0$, hence
$\mathcal{D}^{\hspace{1pt}*}{\hat\phi}$ is nilpotent, too.
  \end{proof}

  \begin{lemm}
  Let ${\hat\phi}:V \rightarrow V$ be a nilpotent linear mapping. Then
  $\exp(\mathcal{D}^{\hspace{1pt}*}{\hat\phi})=\wedge^{*} \exp({\hat\phi})$
  on $\Lambda^{*}V$.
  \end{lemm} 
  \begin{proof}
  We have
  {\allowdisplaybreaks
  \begin{align*}
\exp(\mathcal{D}^{\hspace{1pt}*}{\hat\phi})(x_1 \wedge\ldots\wedge 
x_r)&=\sum_{p \geq 0}
  \frac{1}{p!}\mathcal{D}^{\hspace{1pt}r}{\hat\phi}^p(x_1 \wedge\ldots\wedge 
x_r)\\
  &=\sum_{p \geq0}\frac{1}{p!}(\underset{(i_j \geq 
0)}{\sum_{i_1+\ldots+i_r=p}}
\frac{p!}{(i_1)!\ldots(i_r)!}~{\hat\phi}^{i_1}x_1 
\wedge\ldots\wedge{\hat\phi}^{i_r}x_r)\\
&=\sum_{i_j \geq 0}\frac{1}{(i_1)!\ldots(i_r)!}~{\hat\phi}^{i_1}x_1 
\wedge\ldots\wedge{\hat\phi}^{i_r}x_r\\
&=(\sum_{i_1 \geq 0}\frac{1}{(i_1)!}{\hat\phi}^{i_1}x_1)\wedge\ldots\wedge
(\sum_{i_r \geq 0}\frac{1}{(i_r)!}{\hat\phi}^{i_r}x_r)\\
&=(\exp({\hat\phi}) x_1)\wedge\ldots\wedge(\exp({\hat\phi}) x_r)\\
&=\wedge^{*} \exp({\hat\phi})(x_1 \wedge\ldots\wedge x_r),
   \end{align*}}
   \linebreak
  which yields the result. Note that all sums in the above equalities are 
finite according to the
  previous lemma.
   \end{proof}
   \begin{rema}
   The nilpotency of ${\hat\phi}$ is not necessary in the preceding lemma. 
In fact, one may
   use the definition of
   $\exp:\mathfrak{gl}(\Lambda^*V)\rightarrow\mathrm{GL}(\Lambda^*V)$. More 
precisely,
   define $s:\mathbb{R}\rightarrow\mathrm{GL}(\Lambda^*V)$ by
   $s(t)=\wedge^* \exp({t {\hat\phi}})$. Then one may check that $s$ is the 
1-parameter
   subgroup generated by $\mathcal{D}^{\hspace{1pt}*}{\hat\phi}$ (i.e. 
$\dot{s}(0)=\mathcal{D}^{\hspace{1pt}*}{\hat\phi}$), and we have $s(1)=\wedge^{*} \exp({\hat\phi})$.
   \end{rema}
   \begin{coro}\label{sim}
Let ${\hat\phi}:V \rightarrow V$ be a nilpotent linear mapping, and set ${\hat\epsilon}:={\hat\phi}+\mathrm{id}$.
Then $\wedge^{r} {\hat\epsilon}-\mathrm{id}$ is similar to $\mathcal{D}^{\hspace{1pt}r}{\hat\phi}$ for
  $r \geq 0$.
	\end{coro}
   \begin{proof}
We know from Lemma \ref{lemm:3} that $\exp({\hat\phi})-\mathrm{id}$ is similar to ${\hat\phi}$, 
hence $\exp({\hat\phi})$ is similar to
   ${\hat\phi}+\mathrm{id}={\hat\epsilon}$. So
   \begin{align*}
   \wedge^{r} {\hat\epsilon}-\mathrm{id} &\sim \wedge^{r} \exp({\hat\phi})-\mathrm{id}
   =\exp({\mathcal{D}^{\hspace{1pt}r}{\hat\phi}})-\mathrm{id}
    \sim \mathcal{D}^{\hspace{1pt}r}{\hat\phi} .
      \end{align*}
   \end{proof}

\section{Representation theory of $\mathfrak{sl}(2,\mathbb{C})$}\label{sl_2}
\noindent We refer to Section II.7 of \cite{jeH78} for studying the irreducible representations
of the Lie algebra $\mathfrak{sl}(2,\mathbb{C})$.\\

\noindent Let $\mathfrak{sl}(2,\mathbb{C})$ denote the special linear Lie algebra over 
$\mathbb{C}^2$ defined by $\mathfrak{sl}(2,\mathbb{C}):=\{a \in M_2(\mathbb{C})\mid {\rm Tr}(a)=0\}.$ It is well known that 
$\mathfrak{sl}(2,\mathbb{C})$ is a $3$-dimensional
simple complex Lie algebra. One can check that
$$
\mathfrak{B}:=\{h:=\left(\begin{array}{cc}
1 & 0\\
0 & -1\\
\end{array}
\right),e:=\left(\begin{array}{cc}
0 & 1\\
0 & 0\\
\end{array}
\right),f:=\left(\begin{array}{cc}
0 & 0\\
1 & 0\\
\end{array}
\right)\}$$
is a basis for this Lie algebra. \\

\noindent The following theorem is the foundation of representation theory of semisimple Lie algebras including
$\mathfrak{sl}(2,\mathbb{C})$. It is stated and proved in the Subsection II.6.3 of \cite{jeH78}.
\begin{theo}[Weyl]\label{Weyl}
Every finite dimensional representation of a semisimple
Lie algebra is completely reducible, namely, it can be decomposed into a direct sum of irreducible 
representations.
\end{theo}  

\begin{prop}\label{canonical}
Let $V$ be an \text{$n$-dimensional} complex vector space with a basis 
$\{e_1,\ldots,e_n \}$. Then the following equalities (for $i=1,\ldots,n$)

\begin{itemize}
\item [(a)]$\pi_n(h)e_i=(2i-n-1)e_i$;
\item [(b)]$\pi_n(e)e_i=i(n-i)e_{i+1}$; $(e_{n+1}:=0)$
\item [(c)]$\pi_n(f)e_i=e_{i-1}$; $(e_0:=0)$.
\end{itemize}

\noindent define a representation
$\pi_n:\mathfrak{sl}(2,\mathbb{C})\rightarrow \mathfrak{gl}(V)$.
\end{prop}

\noindent We call $\pi_n$ defined in the previous proposition the \textit{canonical representation of $\mathfrak{sl}(2,\mathbb{C})$ on $V$ associated with the basis $\{e_1,\ldots,e_n\}$}.
\noindent We recall the following theorem from Section II.7 of \cite{jeH78}.
\begin{theo}\label{rep}
Let $\pi_n$ be the representation described above. Then
\begin{itemize}
\item [\textnormal{(i)}]
$\pi_n$ is an irreducible representation of
$\mathfrak{sl}(2,\mathbb{C})$.
\item [\textnormal{(ii)}]
Any \text{$n$-dimensional} irreducible representation of $\mathfrak{sl}(2,\mathbb{C})$
is equivalent to $\pi_n$.
\item [\textnormal{(iii)}]
Let $V$ be a finite dimensional $\mathfrak{sl}(2,\mathbb{C})$-module
and define $$V_\alpha=\{v\in V \mid h.v=\alpha~v\}$$
for $\alpha\in\mathbb{C}$. Then $V$ decomposes into a direct sum of 
irreducible submodules (Weyl),
and in any such decomposition, the number of summands is precisely
$\dim V_0 + \dim V_1$.
\end{itemize}
\end{theo}

\section{Some computer codes and a counterexample} \label{codes}
\noindent To compute the $K$-groups of $\mathcal{C}(\Bbb{T}^n)\rtimes_\alpha
\Bbb{Z}$, one should first compute the kernels and cokernels of the 
following integer matrices 
\begin{equation}
\mathsf{A}_{r}:=\wedge^r \mathsf{A}- \mathsf{I}_{\binom{n}{r}}
\end{equation}
for $r=0,1,\ldots,n$, where $\mathsf{A}$ is the integer matrix corresponding
to $\alpha$ acting on $\Bbb{Z}^n$. We have written two Maple codes to 
obtain this goal. \\

\noindent The first one is an auxiliary procedure called \verb"exterior(r,A)",
which computes the $r$-th  exterior power of a given $n \times n$ integer
matrix $\mathsf{A}$ for $r=1,\ldots,n$ as follows (note that 
$\wedge^0 \mathsf{A}:=\mathsf{I}_1$)
\begin{verbatim}
> exterior:=proc(r,A)
> local n,N,Q,E,i,j;
> n:=linalg[rowdim](A);
> N:=binomial(n,r);
> Q:=combinat[choose](n,r);
> E:=array(1..N,1..N);
> for i from 1 to N do
> for j from 1 to N do
> E[i,j]:=linalg[det](linalg[submatrix](A,Q[i],Q[j]));
> od;
> od;
> RETURN(evalm(E));
> end;
\end{verbatim}
The second code, which calls the first one, is called \verb"po(A)". It lists
the factorized characteristic polynomials of the Smith normal forms 
of $\mathsf{A}_{r}$ for $r=0,1,\ldots,n$ given a matrix $\mathsf{A}$
\begin{verbatim}
> po:=proc(A)
> local r,n,x,p;
> n:=linalg[rowdim](A); 
> print(x);
> for r from 1 to n do
> p:=factor(linalg[charpoly](linalg[ismith](exterior(r,A)-1),x));
> print(p);
> od;
> end;
\end{verbatim}
These factorized polynomials encode the diagonal
entries of the Smith normal forms of the matrices $\mathsf{A}_{r}$ for 
$r=0,1,\ldots,n$. In particular, one can easily
find the kernels and cokernels of these matrices. To see this, observe 
that the kernel of an $m \times m$ integer matrix is isomorphic to
a torsion-free finitely generated abelian group, whose 
rank is the number of zeros on the diagonal of its 
Smith normal form, and the cokernel is isomorphic to 
$\oplus_{j=1}^m \Bbb{Z}_{k_j}$, where $k_1,\ldots,k_m$
are the diagonal entries of the Smith normal form.\\

\noindent The following example computes the $K$-groups of 
$\mathscr{A}_{6,\theta}$, which was mentioned in Remark \ref{rema:2} as a counterexample for
\cite[Proposition 2.17]{rJ86}.
\begin{example}\label{counter}
$K_0(\mathscr{A}_{6,\theta})\cong\Bbb{Z}^{13}$ and 
$K_1(\mathscr{A}_{6,\theta})\cong\Bbb{Z}^{13} \oplus\Bbb{Z}_2$.
\end{example}
\begin{proof}
Using the procedure \verb"po(A)" for the matrix $\mathsf{A}:=\mathsf{S}_6$ as defined 
at the beginning of Section \ref{sec:Anzai} we get
\begin{verbatim}
po(A);
\end{verbatim}
\begin{equation*}
x
\end{equation*} 
\begin{equation*}
x(x-1)^5 
\end{equation*}
\begin{equation*}
x^3(x-1)^{12}
\end{equation*} 
\begin{equation*} 
x^{3}(x - 2)(x - 1)^{16}
\end{equation*} 
\begin{equation*} 
x^3(x-1)^{12}
\end{equation*} 
\begin{equation*} 
x(x-1)^5 
\end{equation*} 
\begin{equation*}
x  
\end{equation*}                       
Consequently, we have the following table representing the kernels
and cokernels of 
$\mathsf{S}_{6,r}:=\wedge^r \mathsf{S}_6-\mathsf{I}_{\binom{6}{r}}$
for $r=0,1,\ldots,6$.

\begin{table}[h]
\begin{center}
\begin{tabular}{|c|c|c|}
\hline
$r$  & $\ker\mathsf{S}_{6,r}$ & $\mathrm{coker}\hspace{2pt}\mathsf{S}_{6,r}$ 
\\ \hline\hline
0 & $\Bbb{Z}$   & $\Bbb{Z}$\\
1 & $\Bbb{Z}$ & $\Bbb{Z}$  \\ 
2 & $\Bbb{Z}^3$ & $\Bbb{Z}^3$   \\ 
3 & $\Bbb{Z}^3$ & $\Bbb{Z}^3 \oplus\Bbb{Z}_2$   \\ 
4 & $\Bbb{Z}^3$ & $\Bbb{Z}^3$   \\ 
5 & $\Bbb{Z}$ & $\Bbb{Z}$   \\ 
6 & $\Bbb{Z}$ & $\Bbb{Z}$  \\ 
\hline
\end{tabular}
\end{center}
\end{table}
%\newpage
\noindent Now, using Theorem \ref{theo:1} we have
\begin{align*}
K_0({\mathscr{A}_{6,\theta}})\cong 
&(\mathrm{coker}\hspace{2pt}
\mathsf{S}_{6,0}\oplus \mathrm{coker}\hspace{2pt}\mathsf{S}_{6,2}
\oplus \mathrm{coker}\hspace{2pt}\mathsf{S}_{6,4}
\oplus \mathrm{coker}\hspace{2pt}\mathsf{S}_{6,6})\oplus
(\ker \mathsf{S}_{6,1}\oplus \ker \mathsf{S}_{6,3}
\oplus \ker \mathsf{S}_{6,5})\\
\cong&(\mathbb{Z}\oplus\mathbb{Z}^3 \oplus\mathbb{Z}^3 
\oplus\mathbb{Z})\oplus(\Bbb{Z}\oplus\Bbb{Z}^3
\oplus \Bbb{Z})=\mathbb{Z}^{13},\\
K_1({\mathscr{A}_{6,\theta}})\cong 
&(\mathrm{coker}\hspace{2pt}
\mathsf{S}_{6,1}\oplus \mathrm{coker}\hspace{2pt}\mathsf{S}_{6,3}
\oplus \mathrm{coker}\hspace{2pt}\mathsf{S}_{6,4})\oplus
(\ker \mathsf{S}_{6,0}\oplus \ker \mathsf{S}_{6,2}
\oplus \ker \mathsf{S}_{6,4}\oplus \ker \mathsf{S}_{6,6})\\
\cong&(\mathbb{Z}\oplus\mathbb{Z}^3 \oplus\Bbb{Z}_2  
\oplus\mathbb{Z})\oplus(\Bbb{Z}\oplus\Bbb{Z}^3
\oplus\Bbb{Z}^3 \oplus \Bbb{Z})=\mathbb{Z}^{13}\oplus\Bbb{Z}_2.
\end{align*}
\end{proof}

\section{Applications to $C^*(\mathfrak{D}_n)$} \label{cocompact}
\noindent This appendix is devoted to some applications of the results of this paper to 
the $C^*$-algebras studied in \cite{kR06}, in which the authors promised to completely classify the 
simple infinite dimensional quotients of the group $C^*$-algebra of their interest $C^*(\mathfrak{D}_n)$ by their $K$-theoretic invariants in a later work. In this context, the discrete group $\mathfrak{D}_n$ is a higher dimensional 
analogue of the discrete Heisenberg group $H_3$, and is defined by
\begin{equation*}
\mathfrak{D}_n=\langle\, x,y_0,y_1,\ldots,y_n\mid x y_0=y_0x,\, y_iy_j=y_j y_i~ \text{for}~ 0 \leq i,j \leq  n,\,
[x,y_j]=y_{j-1}~\text{for}~1\le j \le n\,\rangle,
\end{equation*}    
where $[x,y]:=xyx^{-1}y^{-1}$. The group $\mathfrak{D}_n$ can be represented as a semidirect product 
$\Bbb{Z}^{n+1}\rtimes_\eta \Bbb{Z}$, where the group homomorphism $\eta:\Bbb Z\to\mathrm{GL}(n+1,\Bbb Z)$ is such that $\eta(k)$ is the matrix, whose $(i,j)$-entry is given by $\binom{k}{j-i}$ as defined in \cite{kR06}. This realization allows us 
to study $K$-theory of $C^*(\mathfrak{D}_n)$.
\begin{prop}
$K_i(C^*(\mathfrak{D}_n))\cong K_i(\mathscr{A}_{n+1,\theta})$ for $i=0,1$. 
In particular,
$$\mathrm{rank}\hspace{2pt}K_0(C^*(\mathfrak{D}_n))
=\mathrm{rank}\hspace{2pt}K_1(C^*(\mathfrak{D}_n))=a_{n+1}.$$
\end{prop}
\begin{proof} 
Since $\mathfrak{D}_n \cong \Bbb{Z}^{n+1}\rtimes_\eta \Bbb{Z}$, so
$C^*(\mathfrak{D}_n)\cong C^*(\Bbb{Z}^{n+1})\rtimes_{\tilde{\eta}} \Bbb{Z}
\cong\mathcal{C}(\Bbb{T}^{n+1})\rtimes_{\tilde{\eta}} \Bbb{Z}$ and the
integer matrix corresponding to $\tilde{\eta}$ is the $(n+1)\times(n+1)$ 
matrix $\mathsf{M}_n$ introduced in \cite{kR06}, which is precisely 
the matrix $\mathsf{S}_{n+1}$ defined in Section \ref{sec:rank} to describe 
the linear structure of Anzai transformations on $\Bbb T^{n+1}$.
The rest of the proof follows from the Theorem \ref{theo:1}.
\end{proof}

\noindent As some higher-dimensional analogues of the irrational rotation algebras $A_\theta$, all simple 
infinite-dimensional quotients of $C^*(\mathfrak{D}_n)$ have been classified in Theorem 3.2 of \cite{kR06}. These consist
of the $C^*$-algebras $\mathscr{A}_{n,\theta}$ for some irrational parameter $\theta$, and a few more classes of $C^*$ algebras denoted by $A^{(n)}_i$, which are of the form $\mathcal{C}(\mathbf{Y}_{i}\times\Bbb{T}^{n-i})\rtimes_{\phi_{i}}\mathbb{Z}$ for some suitable finite sets
$\mathbf{Y}_i$ and minimal homeomorphisms $\phi_i$ for $i=1,2,\ldots,n-1$. Then it is proved in Theorem 4.8 of \cite{kR06} that
$A^{(n)}_i \cong M_{C_i}(B^{(n)}_i)$, where $B^{(n)}_i$ is a certain transformation group $C^*$-algebra of some 
affine Furstenberg transformation on $\Bbb T^{n-i}$. We conclude the following results.
\begin{coro}
Let $A$ be a simple infinite dimensional quotient of $C^*(\mathfrak{D}_n)$.
Then $\mathrm{rank}\hspace{2pt}K_0(A)=
\mathrm{rank}\hspace{2pt}K_1(A)=a_{n-i}$
for some $~i \in \{0,1,\ldots,n-1 \}$ that is uniquely determined by the
isomorphism $A \cong \mathcal{C}(\mathbf{Y}_{i}
\times\Bbb{T}^{n-i})\rtimes_{\phi_{i}}\mathbb{Z}$ as in
Theorem 3.2 of \cite{kR06}.
\end{coro}
\begin{proof}
It is proved that $A$ is isomorphic to a matrix algebra over
a Furstenberg transformation group $C^*$-algebra $B^{(n)}_i$ on $\Bbb{T}^{n-i}$ 
for some suitable $i\in\{0,1\ldots,n\}$ \cite[Theorem 4.8]{kR06}. So $K_j(A) \cong K_j(B^{(n)}_i)$ for $j=0,1$. The rest of the proof is clear from the preceding theorem.
  \end{proof}
  
 \noindent  We saw in Proposition \ref{prop:2} that $\{a_n \}$ is a 
strictly increasing
  sequence. Therefore the preceding corollary is a first step towards the 
classification of the simple infinite dimensional quotients of $C^*(\mathfrak{D}_n)$ by means of $K$-theory. 
But as is seen, the rank of the $K$-groups alone can not distinguish the algebras at the same ``level" (i.e. those algebras that
are included in the same class, but with different values of the parameters).
The other powerful $K$-theoretic invariant that helps us do this is the trace invariant, i.e. the
range of the unique tracial state acting on the $K_0$-group.
\begin{prop}
Suppose $A \cong \mathcal{C}(\mathbf{Y}_{i} 
\times\Bbb{T}^{n-i})\rtimes_{\phi_{i}}\mathbb{Z}$
is a simple infinite dimensional quotient of $C^*(\mathfrak{D}_n)$ as in 
Theorem 3.2 in \cite{kR06}. Then $A$ has a unique tracial state $\tilde\tau$
and $\tilde\tau_{*}K_0(A)=\frac{1}{C_i}(\Bbb{Z}+\Bbb{Z}\vartheta_i)$, where 
$C_i=|\mathbf{Y}_{i}|$
and $e^{2 \pi i \vartheta_i}=\zeta_i=(-1)^{C_i+1}\eta_i^{C_i}$ as in Lemma 
4.6 and Theorem 4.8 in \cite{kR06}.
\end{prop}
\begin{proof}
Following Theorem 4.8 in \cite{kR06}, $A$ is isomorphic to 
$M_{C_i}(B^{(n)}_i)=M_{C_i}(\Bbb{C})\otimes B^{(n)}_i$, where 
$B^{(n)}_i$ is a simple Furstenberg transformation group $C^*$-algebra with the irrational
parameter $\zeta_i=(-1)^{C_i+1}\eta_i^{C_i}$. By Corollary \ref{coro:1}, $B^{(n)}_i$ has a unique tracial state $\tau$. Moreover, $\tau_*K_0(B^{(n)}_i)=\Bbb{Z}+\Bbb{Z}\vartheta_i$, where $e^{2 \pi i \vartheta_i}=\zeta_i$
 (see \cite[Theorem 2.23]{rJ86} or Theorem \ref{theo:order}). Thus $A$ has the unique tracial state
  $\tilde\tau=(\frac{1}{C_i}\mathrm{Tr})\otimes\tau$, in which $\mathrm{Tr}$ 
is the usual trace on $M_{C_i}(\Bbb{C})$, and so
  $\tilde\tau_{*}K_0(A)=\frac{1}{C_i}(\Bbb{Z}+\Bbb{Z}\vartheta_i)$ 
\cite[Lemma 3.5]{rJ86}.
\end{proof}
\noindent Finally, we can characterize all simple infinite-dimensional quotients of $C^*(\mathfrak{D}_n)$.
\begin{prop}
$\mathscr{A}_{n,\theta}\cong\mathscr{A}_{n',\theta'}$ if and only if $n=n'$ and there 
exists an integer $k$ such that $\theta=k \pm\theta'$. More generally, let 
$A_i^{(n)}\cong
\mathcal{C}(\mathbf{Y}_{i} \times\Bbb{T}^{n-i})\rtimes_{\phi_{i}}\mathbb{Z}$
be a simple infinite dimensional quotient of $C^*(\mathfrak{D}_n)$ with the
structure constants $\lambda,\mu_1,\ldots,\mu_i$ as on p. 165-166
of \cite{kR06}, and let $A_{i'}^{(n')}\cong
\mathcal{C}(\mathbf{Y'}_{i'} 
\times\Bbb{T}^{n'-i'})\rtimes_{\phi'_{i}}\mathbb{Z}$
be a simple infinite dimensional quotient of $C^*(\mathfrak{D}_{n'})$ with 
the structure constants $\lambda',\mu'_1,\ldots,\mu'_{i'}$.
Suppose that $C_i=|\mathbf{Y}_{i}|$ and $C'_{i'}=|\mathbf{Y'}_{i'}|$. Then 
$A_i^{(n)}\cong A_{i'}^{(n')}$ if and only if
$n-i=n'-i'$, $C_i=C'_{i'}$ and
$$\lambda^{\binom{{C_i}}{i+1}}\mu_1^{\binom{{C_i}}{i}}
                 \mu_2^{\binom{{C_i}}{i-1}}\ldots\mu_{i}^{C_i}=                 
\lambda'^{\binom{{C'_{i'}}}{i'+1}}{\mu'_1}^{\binom{{C'_{i'}}}{i'}}                 
{\mu'_2}^{\binom{{C'_{i'}}}{i'-1}}\ldots{\mu'_{i'}}^{C'_{i'}}$$
                 or
                 $$\lambda^{\binom{{C_i}}{i+1}}\mu_1^{\binom{{C_i}}{i}}
                 \mu_2^{\binom{{C_i}}{i-1}}\ldots\mu_{i}^{C_i}=                 
(\lambda'^{\binom{{C'_{i'}}}{i'+1}}{\mu'_1}^{\binom{{C'_{i'}}}{i'}}                 
{\mu'_2}^{\binom{{C'_{i'}}}{i'-1}}\ldots{\mu'_{i'}}^{C'_{i'}})^{-1}.$$
\end{prop}
\begin{proof}
Use the previous proposition and the fact that $\{a_n \}$ is a strictly increasing 
sequence
(see Proposition \ref{prop:2}). Note that 
$$\zeta_i=(-1)^{C_i+1}\eta_i^{C_i}=
\lambda^{\binom{{C_i}}{i+1}}\mu_1^{\binom{{C_i}}{i}}
                 \mu_2^{\binom{{C_i}}{i-1}}\ldots\mu_{i}^{C_i}$$
                 by the last equation on p. 171 of \cite{kR06}.
\end{proof}
\begin{rema}
Note that $C_i=|\mathbf{Y}_{i}|$ is completely determined by the structural 
constants
$\lambda,\mu_1,\ldots,\mu_{i-1}$ (which are roots of unity). More precisely, by calculations on
p. 165-166 of \cite{kR06}, we have
$$C_i=\min\{r \in\Bbb{N}\mid\lambda^r=\lambda^{\binom{r}{2}}\mu_1^r=\ldots=
\lambda^{\binom{r}{i}}\mu_1^{\binom{r}{i-1}}
                 \mu_2^{\binom{r}{i-2}}\ldots\mu_{i-1}^r=1\}.$$
For an explicit example in lower dimensions, see \cite[Lemma 5.4]{pM97}.
\end{rema}

\end{document}